\documentclass[reqno,11pt]{article}
\usepackage{amssymb,amsthm,amsmath,amsfonts,mathrsfs}
\usepackage{graphicx}
\usepackage[top=1in,bottom=1in,left=0.9in,right=0.9in]{geometry}
\usepackage[latin1]{inputenc}
\usepackage[T1]{fontenc}
\usepackage{hyperref}
\usepackage[english]{babel}
\usepackage{color}
\usepackage{esint}

\newtheorem{prop}{Proposition}[section]
\newtheorem{lem}{Lemma}[section]

\newtheorem{thm}{Theorem}[section]
\newtheorem{rmq}{Remark}[section]

\numberwithin{equation}{section}

\newcommand{\be}{\begin{equation} \label}
\newcommand{\ee}{\end{equation}}
\newcommand{\R}{\mathbb{R}}
\newcommand{\N}{\mathbb{N}}
\newcommand{\eps}{\varepsilon}
\newcommand{\vap}{\varphi}

\newcommand{\cb}{\color{blue}}

\author{Loth}

\begin{document}

	\begin{center}
		\section*{Local existence and blow-up behavior for the diffusive Hamilton-Jacobi equation in a half-space with unbounded initial data}
		$ $
	\end{center}
	
	\begin{center}
		  Loth Damagui CHABI
\footnote{Université Sorbonne Paris Nord, LAGA, CNRS, UMR 7539, F-93430, Villetaneuse, France\\
{\it Email address}: chabi@math.univ-paris13.fr}
		$ $
		
	\end{center}
\begin{abstract}

We study the local existence and blow-up behavior of solutions, with possible growth at space infinity, 
for the diffusive Hamilton-Jacobi equation $u_t-\Delta u=|\nabla u|^p$ ($p>1$), in a half-space with Dirichlet boundary conditions.
Under optimal, polynomial growth assumptions on the initial data, characterized by the critical exponent $p/(p-1)$,
we first prove local existence and uniqueness of a maximal classical solution and establish a blow-up alternative.
This is complemented by a nonexistence result showing that $p/(p-1)$ is the sharp threshold for admissible growth at infinity. 
Next, for initial data with subcritical growth, we show global-in-time existence for $1<p\le2$ and that finite-time blow-up can occur only through gradient blow-up for $p>2$, 
and we derive sharp upper and lower estimates on the gradient,
including a precise type~II boundary blow-up rate and a dichotomy in its behavior near the singular time.
Finally, in the case $p>2$, for any type~II blow-up solution, we prove convergence after rescaling) to an explicit one-dimensional profile and providing a refined description of the asymptotic singularity formation.

\smallskip

\noindent{\bf Keywords.} Diffusive Hamilton-Jacobi equation, local existence, nonexistence, blow-up behavior.

\noindent{\bf MSC Classification.} 35F21, 35A01, 35K55, 35B45

\end{abstract}

\section{Introduction}
 In this paper we are concerned with 
 the diffusive Hamilton-Jacobi equation
\be{localtheoryeqE0}
u_t-\Delta u=|\nabla u|^p,
\ee
in a half-space with Dirichlet boundary conditions, where $p>1$. 

First of all, let us recall that the corresponding Cauchy-Dirichlet problem
\begin{equation}\label{localtheory01eqE1b}
	\begin{cases}
	u_t-\Delta u=|\nabla u|^p,&x\in\Omega,\ t>0,\\
	u=0,&x\in \partial\Omega,\ t>0,\\
	u(0,x)=u_0(x),&x\in \Omega,
	\end{cases}
	\end{equation}
where $\Omega\subset\R^n$ is a smooth bounded domain,
arises in stochastic control. Namely, for $u_0\in C_0(\overline\Omega)$,
problem~\eqref{localtheory01eqE1b} has a unique global viscosity solution, which gives the value function of the optimal control problem associated with the stochastic differential system $dX_s =\alpha_s ds+dW_s$, with cost function $|\alpha_s|^{p/(p-1)}$
and zero pay-off at the exit time.
Here $(W_s)_{s>0}$ is a standard Brownian motion, 
 $\alpha_s$ is the control, and $u_0$ represents the distribution of rewards
(see \cite{BaLio04} and cf.~\cite{LL, BaLio04, BarCD, FlMS, AttSou} for more details).
As another motivation,  \eqref{localtheoryeqE0} corresponds to the so-called deterministic KPZ equation
which, along with its stochastic version,
arises in a well-known model of surface growth by ballistic deposition (cf.~\cite{KPZ86, KS88}
 and see \cite{Ha1, Ha2} for far-reaching developments in the stochastic case).
Also, it can be seen as one the simplest model parabolic problems with first order nonlinearity and, from the point of view of nonlinear parabolic theory, it is thus important to understand its properties 
(cp.~for instance with the extensively studied equation with zero order nonlinearity $u_t-\Delta u= u^p$; see \cite{quittner2019superlinear}).

\vspace{2pt}

 Although \eqref{localtheory01eqE1b} admits a global viscosity solution, classical solutions may lose regularity in finite time. For smooth initial data $u_0\in C^1(\bar\Omega)$ vanishing on the boundary, there exists a unique maximal classical solution, defined on $(0,T)$. 
For $p\in(1,2]$, all solutions are global. When $p>2$,   sufficiently large nonnegative initial data lead to finite-time blow-up, characterized by the unboundedness of the gradient while the solution itself remains bounded. 
This phenomenon, known as gradient blow-up (GBU), has been extensively studied see e.g.  \cite[Section~40]{quittner2019superlinear} and references therein, including blow-up criteria, location, rates, profiles, continuation after blow-up, and related qualitative properties, within a vast literature on this equation and related models. 
Blow-up for the bounded domain case is by now rather well understood. Roughly speaking it can be always described as ''type II boundary
gradient blow-up'', and the situation is essentially similar in unbounded domains if the initial data is bounded (along with its gradient). 
Indeed,
 GBU can occur only on the boundary, see e.g. \cite{SZ}.
Moreover, it was proved in \cite{GuoHu,PS_jmpa18} that the GBU rate is always type~II, i.e. does not coincide with the one predicted by the self-similar invariance of the problem. More precisely, for positive initial data in $BC^1(\overline\Omega)$, any non-global solution satisfies
\be{localtheoryfinal3}
\|\nabla u(\cdot,t)\|_{L^\infty(\Omega)}\ge C (T-t)^{-1/(p-2)}\quad\hbox{as }\quad t\to T_{\max}.
\ee

As far as we know, the only works on problem \eqref{localtheory01eqE1b} in unbounded domains studying solutions growing at space infinity are \cite{chso2,GGK,MS1}.
The paper \cite{GGK} considered the one-dimensional
Cauchy problem ($\Omega=\R$), and proved local well-posedness for continuous initial data with an upper
polynomial bound of order $p/(p-1)$. Self-similar solutions with profiles growing at this order were
constructed in \cite{GGK} and \cite{chso2}, respectively for $\Omega=\R$ and $\Omega=\R^n_+$. 
We also refer the reader to \cite{FankKwak} for related information on self-similar solutions of \eqref{localtheoryeqE0} in the case $\Omega=(0,\infty)$.
 On the other hand, for problem \eqref{localtheory01eqE1b} with $\Omega=(0,\infty)$, a local existence-uniqueness result
in a class of initial data growing at most linearly at infinity can be found in \cite[Proposition~3.17]{MS1}.

\vspace{2pt}

 The first main goal of this paper is to address the so far unexplored question of the blow-up behavior for
equation \eqref{localtheoryeqE0} in the unbounded domain case with unbounded initial data. Our study will lead to results which
  display interesting differences with the usual cases. We
shall concentrate on the case of the simplest unbounded domain with nonempty boundary, namely the half-space (this is a natural choice, rather than the
whole space, in view of the above mentioned boundary blow-up phenomena).

In order to carry out this study, it is obviously necessary to first have a precise local well-posedness theory for unbounded initial data. This is also an interesting topic in itself, with the natural question of whether a maximal admissible growth appears. Since only partial results seem to be available in this direction, a second main goal of this paper is to fill this gap.

\section{Main results}
	 \subsection{Local well-posedness   with optimal class of initial data}
In this paper,
 for $T>0$, we denote 
$$
H:=\R^n_+,\quad Q_T:=H\times(0,T)
$$
and consider the problem
\begin{equation}\label{localtheoryeqE1b}
	\begin{cases}
	u_t-\Delta u=|\nabla u|^p,&x\in H,\ t\in(0,T),\\
	u=0,&x\in \partial H,\ t\in(0,T),\\
	u(0,x)=u_0(x),&x\in H.
	\end{cases}
	\end{equation}
 By a solution  $u$ of \eqref{localtheoryeqE1b}, we mean a function $u\in C^{2,1}(Q_T)\cap C(\overline H\times [0,T))$ which verifies \eqref{localtheoryeqE1b} in the classical sense.
  We denote 
  
 $$
  \beta=1/(p-1),\quad d_p:=\beta^{\beta},\quad c_p:=(1-\beta)^{-1}d_p\quad \hbox{ for  $p>2$}.
$$

For given $0\le \eta<R$, we also denote
$$
\Gamma_{\eta,R}:=\{x\in\R^n,\ \eta< x_n< R\}\quad \hbox{ and }\quad \Gamma_R:=\Gamma_{0,R}.
$$
Let $w=(1+x_n)^{\beta+1}$. We define
$$
C^0_\mathcal{U}(\bar H)=\Big\{v\in C(\bar H);\ v\ge 0,\ \sup_{\Gamma_r}v\to 0\ \hbox{as }r\to 0\Big\},\quad L_{w}^\infty(H)=\Big\{ v\in L_{loc}^\infty(\bar H),\ v\ge0,\ \sup_{H} \Big(\frac{v}{w}\Big)<\infty\Big\}.
$$
$C^0_{\mathcal{U}}(\overline Q_T)$ and $L_w^\infty(Q_T)$ are defined similarly with $H$, $\Gamma_r$ replaced by $Q_T$, $\Gamma_r\times(0,T)$, respectively, and we also denote 
\be{localtheorydefset}
Y_T:=C^0_\mathcal{U}(\overline Q_T)\cap L_{w}^\infty(Q_T)\cap C^{2,1}(Q_T).
\ee
Our optimal class of initial data, with possible growth at $\infty$ and singularity on the boundary, is defined by 
$$
\begin{aligned}
&X:= C^0_\mathcal{U}(\bar H)\cap L_{w}^\infty(H) \quad \hbox{if $1<p\le 2,\ $ and }\\
&X_1=\Big\{v\in X:\ \limsup_{x_n\to 0}\Big(\sup_{x'\in\R^{n-1}}x_n^{\beta-1}v(x',x_n)\Big)<c_p\Big\}\quad \hbox{if $p>2$.}
\end{aligned}
$$

Our first main theorem is the local well-posedness of \eqref{localtheoryeqE1b} for initial data in $X_1$, with the corresponding blow-up alternative.
\begin{thm}[Maximal solution for $p>2$]\label{localtheorymaxsol}
Let $p> 2$, and $u_0\in X_1$. There exist $T_{\max}:=T_{\max}(u_0)\in (0,\infty]$ and a unique $u$ such that:

\smallskip

\noindent $(i)$ For any $T\in (0,T_{\max})$,  $u\in Y_T$ is a solution of \eqref{localtheoryeqE1b}.

\smallskip

\noindent $(ii)$ For all $R>0$, we have 
\be{localtheoryconstsol2}
|\nabla u|\in L^\infty\Big(\bar \Gamma_R\times (\eps,T)\Big),\quad \hbox{ for any }\ 0<\eps<T<T_{\max}.
\ee

\noindent $(iii)$ If $T_{\max}<\infty$ then 
$$
\|u(\cdot,t)\|:=\|u(\cdot,t)\|_{L^\infty_w(H)}+\|\nabla u(\cdot,t)\|_{L^\infty(\Gamma_1)}\to \infty\quad\hbox{as }\ t\to T_{\max}.
$$
\end{thm}
Assertion $(iii)$ means that, if $u$ is nonglobal, then either $u$ blows up in the $L^\infty_w$ norm,
or else GBU occurs at finite distance from the boundary. 
Concerning the case $p\in (1,2]$, we obtain an analogue of the above properties for {  the} larger class of initial data { $X$}.

\begin{thm}[Maximal solution  for $1<p\le2$]\label{localtheorymaxsol1p2}
Let $1<p\le 2$, and $u_0\in X$ with $u_0\ge0$. There exist $T_{\max}:=T_{\max}(u_0)\in (0,\infty]$ and a unique $u$ such that assertions $(i)$ and $(ii)$ are valid. Moreover
 if $T_{\max}<\infty$ then 
$$
\|u(\cdot,t)\|_{L^\infty_w( H)}\to \infty\quad\hbox{as }\ t\to T_{\max}.
$$
\end{thm}

The uniqueness part of Theorems~\ref{localtheorymaxsol}-\ref{localtheorymaxsol1p2} is a consequence of the following result which yields local uniqueness {  for} initial data {  in $X$}.

\begin{prop}[Local uniqueness]\label{localtheoryuni}
 Let $p>1$ and  $u_0\in X$. For any $T>0$, problem \eqref{localtheoryeqE1b} admits at most one solution in $Y_T$. 
\end{prop}

Let us next examine the sharpness of the class $X_1$ in Theorem~\ref{localtheorymaxsol}. We first note that $X_1$ is optimal in terms of boundary behavior. Indeed, for $u_0(x)=c_px_n^{1-\beta}$, we see that $u(x,t)=u_0(x)$ solves \eqref{localtheoryeqE1b}. Moreover, it is the only solution in view of Proposition~\ref{localtheoryuni}, so that \eqref{localtheoryeqE1b} has no solution belonging to $C^1(\bar  H)$ for any $t>0$, hence \eqref{localtheoryconstsol2} cannot hold.\\
 Our next theorem is a local nonexistence result, which shows that $X_1$ and $X$ are sharp in terms of behavior at infinity.

\begin{thm}[Non-existence]\label{localtheorynonexistence}
 Let $p>1$ and $u_0\in C(\bar  H)$ be such that $x_n^{-\beta-1} u_0(x)\to \infty$ as $x_n\to \infty$ uniformly with respect to $x'\in \R^{n-1}$, and let $T>0$.
Then there is no nonnegative classical solution to problem \eqref{localtheoryeqE1b} in the class $C^{2,1}(Q_T) \cap C(\bar Q_T)$.
\end{thm}

\begin{rmq}\label{localtheoryrem1} 
$(i)$ We note that our results significantly extend the classical well-posedness framework. In the case of unbounded domain $ H$ classical theory assumes  $u_0\in BC^1(\bar  H)$, namely boundedness of both $u_0$ and $|\nabla u_0|$. Here our theorem instead allows $u_0(x)\lesssim  x_n^{1-\beta}$ near $x_n=0$ and $u_0(x)\lesssim  x_n^{1+\beta}$ at $x_n\to \infty$ and this is optimal.

 \noindent $(ii)$  We have restricted our study to nonnegative initial data and solutions, 
for simplicity and since this represents the main issue, in view of the sign of the nonlinear term. 
However this restriction could be relaxed by simple modifications.

\noindent $(iii)$ 
For $p>2$, similar to the bounded domain case, we have $T_{\max}(u_0)<\infty$ for sufficiently large initial data. For instance, for any $\phi\in X_1$, this holds for $u_0=\lambda \phi$ and $\lambda$ large, as a consequence of blow-up results in bounded domains and comparison arguments (see e.g. the proof of \cite[Lemma~2.3]{LS10}).
\end{rmq}

\subsection{Blow-up behavior}
We now turn to the description of the blow-up behavior of solutions.
 In view of the scaling property of the equation, we say that a nonglobal solution of \eqref{localtheoryeqE0} is of type~I if, for any $R>0$, 
$$
 \limsup_{t\to T_{\max} }\ (T_{\max}-t)^{\frac{1}{p-1}}\| \nabla u(\cdot,t)\|_{L^\infty(\Gamma_R)}<\infty.
$$
On the contrary, it is said to be of type II, if there exists $R_0>0$ such that 
$$
 \limsup_{t\to T_{\max}}\  (T_{\max}-t)^{\frac{1}{p-1}}\| \nabla u(\cdot,t)\|_{L^\infty(\Gamma_{R_0})}=\infty.
$$
We note that for $u_0\in X_1$ 
 with critical growth
 $u_0\simeq C x_n^{\beta+1}$ as $x_n\to \infty$, Theorem~4 in \cite{chso2} gives an example of a {  type~I} $L^\infty$  blow-up solution such that the solution (as well as its gradient) blows up at every point in $ H$. Namely, it provides a self-similar solution 
\be{localtheory1dsol} 
v(x,t)=(T-t)^\gamma\phi(x_n/\sqrt{T-t}) \quad \hbox{in }\quad (-\infty,T),
\ee
 hence $v(x,0)=T^\gamma \phi(x_n/\sqrt{T})$, where $\gamma=\frac{p-2}{2(p-1)}$ and the profile $\phi\in C^2([0,\infty))$ satisfies $\phi(0)=0$, $\phi'>0$ on $[0,\infty)$ and 
 $\displaystyle\lim_{y\to\infty} \frac{\phi(y)}{y^{\beta+1}}=L>0$.
 Moreover, for each $x\in  H$,
 \be{localtheoryexam}
 \begin{cases}
 v(x,t)\sim L(T-t)^{-\beta}x_n^{\beta+1},
 
 \vspace{.2cm}\\
 
 |\nabla v(x,t)|\sim L(\beta+1)(T-t)^{-\beta}x_n^{\beta},
 \end{cases}\quad\hbox{as }\ t\to T.
 \ee
In the subcritical growth case, i.e. when $u_0\in X_1$ satifies $u_0\lesssim x_n^\gamma$ for $x_n\gg 1$ with $\gamma<\beta+1$, we obtain a refined description of the solution and its gradient which exhibits a completely different behavior.
The following theorem in particular shows that the corresponding solution remains bounded on any strip $\Gamma_R\times[0,T)$ and that finite-time blow-up necessarily leads to type~II boundary gradient blow-up (GBU) on $\Gamma_R$. Moreover, we derive upper and lower bounds for the GBU.

 \begin{thm}\label{localtheoryonlyGBUpossible}
 Let $p>1$, $\gamma\in (0,\beta+1)$, let $u_0\in X$ satisfy
 \be{localtheoryGBUassump}
 0\le u_0\le M(1+x_n)^\gamma\quad\hbox{for some $M>0$},
\ee
 let $u$ be the maximal solution of \eqref{localtheoryeqE1b} and set $T=T_{\max}(u_0)$.

\noindent $(i)$ 
Then there exists a constant $C_1:=C_1(M,p)>0$  
such that: 
 \be{localtheoryC1-1double}
 u(x,t)\le  C_1(t+1)^q(1+x_n)^\gamma,\quad (x,t)\in  H\times[0,T),
\ee
where $q=(p+\gamma)/(p-\gamma(p-1))>0$.
In particular, if $p\in(1,2]$, then $T=\infty$.

\vspace{4pt}

\noindent $(ii)$ If $p>2$ and $T<\infty$, then there exists a constant $C_2>0$, depending on $u_0$ through $M$, $\gamma$ and $T$, such that
\be{localtheorygradlowerbound}
 \sup_{\Gamma_1}|\nabla u(x,t)|\ge C_2(T-t)^{-1/(p-2)},\quad t\in (T/2,T)
 \ee  
 and, for any $\eps>0$, there exists $C_\eps>0$ such that
 \be{localtheoryuppergradbound}
 |\nabla u(x,t)|\le (1+\eps)d_p x_n^{-\beta}+C_\eps, \quad (x,t)\in \Gamma_1\times [T/2,T).
 \ee
 \end{thm}

\begin{rmq}
In a bounded domain, it is known \cite{FPS2020} that any non-global solution satisfies  \eqref{localtheoryuppergradbound} everywhere, where $x_n$ is replaced by $dist(x,\partial \Omega)$, and this is optimal.
\end{rmq}

 \begin{rmq}
Under the assumption of Theorem~\ref{localtheoryonlyGBUpossible}(ii), estimate \eqref{localtheoryuppergradbound} is complemented by
$$
|\nabla u(x,t)|\le C x_n^{\gamma-1}+\tilde C, \quad (x,t)\in \big( H\setminus\Gamma_1\big)\times[T/2,T),
$$
 for some constants $C, \tilde C > 0$.
This is a direct consequence of \eqref{localtheoryC1-1double} and the Bernstein estimate \eqref{localtheorygrabound}. Moreover,  by \eqref{localtheoryuppergradbound} for any $\eps>0$  we have
$$
 u(x,t)\le (1+\eps)c_p x_n^{1-\beta}+C_\eps x_n, \quad (x,t)\in \Gamma_1\times [T/2,T).
$$
\end{rmq}
We note that the critical exponent $\gamma=1+\beta$ coincides with that obtained in \cite{GGK} for $ H=\R$. 
The property \eqref{localtheorygradlowerbound} is the analogue of \eqref{localtheoryfinal3} in the case of a bounded domain.
We refer to Theorem 2.1(ii) in \cite{MS1} for the existence of a blow-up solution satisfying the two sided version of \eqref{localtheorygradlowerbound} in the half-line case, i.e., $ H=(0,\infty)$.

 We next give the following bubbling property of type II solutions.

\begin{thm}\label{localtheoryapplication}
Let $p>2$ and $u$ be the maximal solution of \eqref{localtheoryeqE1b} corresponding to $u_0\in X_1$ and set $T=T_{\max}(u_0)$.
If $u$ is of type II, then there exists a sequence $(x_j,t_j)_j\subset H\times[T/2,T)$ with $(x_{j,n},t_j)\to (0,T)$ such that  
 $$
\lambda_j^{\beta}\nabla u(x_j+\lambda_j x,t_j+\lambda^2_j t)\to d_p(x_n+\beta)^{-\beta} e_n,\quad \hbox{locally uniformly on $\bar  H$ as } j\to \infty, 
 $$
 where $\lambda_j=|\nabla u(x_j,t_j)|^{1-p}$ and $e_n:=(0,\cdots,0,1)\in \R^n$.
 \end{thm}

It is known that blow-up in a bounded domain is always  GBU of type~II (see \cite{GuoHu,PS_jmpa18}). The situation is completely different in a half-space, as there exist both type~II GBU solutions controlled by the singular stationary profile (cf.~Theorem~\ref{localtheoryonlyGBUpossible}(ii)), and  type~I $L^\infty$ blow-up solutions of the form \eqref{localtheoryexam} from \cite[Theorem~4]{chso2}.
The following result roughly says that these are, at least locally, the only two possibilities. Namely, at any point $(x,t)$, the gradient of the solution either behaves in a type~I manner, or is controlled by the singular stationary profile.

\begin{thm}\label{localtheorydichotomie}
 Let $p>2$ and $u$ be the maximal solution of \eqref{localtheoryeqE1b} corresponding to $u_0\in X_1$  and set $T=T_{\max}(u_0)<\infty$.
For all $\eps>0$, there exists a constant $C=C(n,p,\eps)$ such that for all $(x,t)\in  H\times[T/2,T)$:
\be{localtheorydico1}
|\nabla u(x,t)|\le Cx_n^\beta(T-t)^{-\beta} \quad
\hbox{or}\quad 
|\nabla u(x,t)|\le (1+\eps)d_px_n^{-\beta}+C.
\ee
\end{thm}

The proof of Theorems~\ref{localtheoryapplication} and \ref{localtheorydichotomie} will make use a recent Liouville-type theorem from \cite{chso2}. The latter shows that any entire solution ($u\in C^{2,1}( H\times\R$) of \eqref{localtheoryeqE1b}  is stationary and only depends on $x_n$ (we refer to \cite{chso2} for other Liouville-type theorems concerning \eqref{localtheoryeqE0}). Theorem~\ref{localtheoryapplication} is an analogue of \cite[Theorem~5.1]{PQ} for the Fujita type equation
\be{localtheoryparabclass}
u_t-\Delta u= u^p.
\ee
Namely, for $p\ge p_L:=1+6/(n-10)_+$, any radial type~II solutions follows a bubling behavior: there exists a $t_j\to T_{\max}$ such that 
\be{localtheorybubl}
\lambda_j^{2/(p-1)} u(\lambda_j r, t_j)\to \phi_1(x_n),\quad \hbox{as } j\to \infty,
\ee
 where $\lambda_j=\| u(\cdot,t_j)\|_\infty^{(1-p)/2}$ and $\phi_1$ is the radial positive
steady state of \eqref{localtheoryparabclass} normalised by $\phi(0)=1$.
The proof of \cite{PQ} is based on a Liouville-type theorem \cite[Theorem~1.6]{PQ}, which asserts that any radial bounded solution of \eqref{localtheoryparabclass} in $\R^n\times\R$ is necessarily stationary (this property fails for $p<p_L$, cf. \cite{FY}).
The property \eqref{localtheorybubl}  is actually known for all $p\ge p_S$ (\cite{MM04}) but under an additional zero number assumption on $u_0$.
In the subcritical range $p<p_S$, we refer to \cite{BV,merle1998optimal,MZ00,PQS,Qu} for other Liouville-type theorems and their applications to blow-up behavior.

\subsection{Ideas of proofs}

 We begin by showing uniqueness of solutions to \eqref{localtheoryeqE1b} in the class $Y_T$ (Proposition~\ref{localtheoryuni}) in Section~\ref{localtheoryun1}, as a consequence of a variant of the comparison principle Proposition~\ref{localtheoryuniqueness} below,  which allows solutions up to the critical growth at space infinity.
The latter is proved by combining the Bernstein-type estimate \eqref{localtheorygrabound} with a penalization argument at space infinity, inspired by \cite{GilGK,GGK}.

In Section~\ref{localtheoryNE}, we derive our local nonexistence theorem by means of comparison with special backward self-similar solutions from \cite{chso2}
and a translation argument.

We next prove, in Section ~\ref{localtheoryExistencesection}, the existence of a solution on a small time interval $(0,T)$, which belongs to $Y_T$ and satisfies \eqref{localtheoryconstsol2}, 
where we carefully estimate the dependence of $T$ on the initial data.
Theorems~\ref{localtheorymaxsol}-\ref{localtheorymaxsol1p2}, including the blow-up alternatives, are derived as a consequence of this and Proposition~\ref{localtheoryuni}.

 Section~\ref{localtheoryBA} is devoted to the proof of Theorem~\ref{localtheoryonlyGBUpossible}. We derive \eqref{localtheoryC1-1double} in Theorem~\ref{localtheoryonlyGBUpossible} by a comparison argument using supersolutions of the form $M(t+1)(A(t)+x_n)^\gamma$. 
  The optimal lower type~II blow-up rate estimate \eqref{localtheorygradlowerbound} is proved by adapting to the present context
the semigroup arguments from \cite{PS_jmpa18} (see also the proof of \cite[Theorem~40.18*]{quittner2019superlinear}).
The proof of \eqref{localtheoryuppergradbound} is done in three steps. We first prove a uniform bound on $u_t$ in finite strips by combining parabolic regularity and maximum principle arguments. We next obtain a rough version of \eqref{localtheoryuppergradbound} by a Bernstein type argument using the bound on $u_t$.
Based on this, the sharp version of \eqref{localtheoryuppergradbound} is then deduced by a rescaling argument making use of the elliptic Liouville-type (one-dimensionality) theorem from \cite{FPS2020}.

In Section~\ref{localtheoryAPL}, we provide the proofs of Theorems~\ref{localtheoryapplication}-\ref{localtheorydichotomie}. These results rely on a priori estimates and a compactness argument, combined with the parabolic Liouville-type theorem \cite[Theorem~3]{chso2}. The proof of Theorem~\ref{localtheorydichotomie} adapts the approach of \cite[Section~3]{FPS2020} to our setting, making use of the a priori estimate \eqref{localtheoryusful1} and \cite[Theorem~3]{chso2}.
Similarly, Theorem~\ref{localtheoryapplication} follows from Proposition~\ref{localtheorythmancient3} together with the doubling lemma from \cite{PQS2}.

\section{Local theory: proof of Theorems~\ref{localtheorymaxsol}-\ref{localtheorynonexistence} and Proposition~\ref{localtheoryuni}} \label{localtheoryLT}

In this section, we denote
$$
L u:=\partial_t u -\Delta u.
$$
 We also denote $BC_0^1(\bar  H)$ 
 the space of functions $u\in C^1(\bar H)$ such that $u$ and $\nabla u$ are bounded in $ H$ and $u=0$ on $\partial H$.
\subsection{Uniqueness: proof of Proposition~\ref{localtheoryuni}}\label{localtheoryun1}
In this section we prove the uniqueness of the solution to \eqref{localtheoryeqE1b} in the class $Y_T$, defined in \eqref{localtheorydefset}.
Proposition~\ref{localtheoryuni} will follow from the following comparison principle.
\begin{prop}\label{localtheoryuniqueness}
Let $u \in Y_T$ be a solution of \eqref{localtheoryeqE1b},
and let $v \in C(\bar H\times[0,T))\cap C^{2,1}( Q_T)$ be a nonnegative function such that $Hv\ge |\nabla v|^p$ in $Q_T$.
Assume that 
 $u(\cdot,0)\le v(\cdot,0)$ in $ H$.
Then $u\le v$ in  $Q_T$.
\end{prop}

\begin{rmq}\label{localtheoryuni+1}
$(i)$ We stress that the conclusion of Proposition~\ref{localtheoryuniqueness} remains valid whenever $u\in C_{\mathcal{U}}^0(Q_T)\cap C^{2,1}( Q_T)$ satisfies $Hu\le |\nabla u|^p$ and that $u\in L_w^\infty (Q_T)$ is strengthened to
$$
|\nabla u(x,t)|\le C(1+x_n)^\beta,\quad (x,t)\in  Q_T.
$$
$(ii)$ For any any $R\in (0,\infty)$, it is easily checked that the argument of the proof of Proposition~\ref{localtheoryuniqueness} can be adapted  whenever $ H$ is replaced by $\Gamma_R$, assuming that $u\in C^{2,1}(Q_T)$ and $u,\nabla u\in L^\infty(Q_T)$ instead of $u\in Y_T$, and that $u\le v$ on the parabolic boundary of $\Gamma_R\times(0,T)$.
\end{rmq}

\begin{proof}[Proof of Proposition~\ref{localtheoryuniqueness}] This will be done in two steps. First, we prove the existence of a small time $\tau>0$ such that $u\le v$ in $ H\times[0,\tau]$. Next, denoting by $\tilde{T}$ the supremum of such $\tau$, we show that $\tilde{T}=T$ using a continuation argument.

Assume by contradiction that for some $\tau\in(0,T)$, 
\be{localtheorycontrad1}
\sigma:=\sup_{ Q_\tau} (u-v)\in (0,\infty].
\ee
 Since $u\in Y_T\subset C_{\mathcal{U}}^0(\bar Q_T)$ and $v\ge 0$, there exists $R_0>0$ such that
\be{localtheory3.1'}
\sup_{\Gamma_{R_0}\times(0,\tau)} (u-v)\le \min\{1,\sigma/8\}=\tilde{\sigma}.
\ee
 Consequently, there exist $x_0\in H\setminus \Gamma_{R_0}$ and $t_0\in(0,\tau]$ such that
\be{localtheory3.1''}
(u-v)(x_0,t_0)\ge 3\tilde\sigma.
\ee
On the other hand, by using \eqref{localtheorygrabound} in Theorem~\ref{localtheorypropBern} and $u\in Y_T\subset L_w^\infty(Q_\tau)$, we have 
\be{localtheorygradientu}
|\nabla u(x,t)|\le C_1(n,p,R_0) t^{-1/p} (1+x_n)^\beta\quad \hbox{ for } x_n\ge R_0.
\ee

 Consider the perturbed function 
$$
Z_\eps=u-v-h_\eps
$$
where $\eps>0$, 
$$
h_{\eps}(x',x_n,t)=M\theta\Big(\eps(1+x_n)\Big)(1+x_n)^{1+\beta}e^{\mu t^{1/p}}+\eps\log(1+|x'|^2)\ge 0,\quad M=\max\Big\{1,\|u\|_{L^\infty_w( Q_\tau)}\Big\},
$$
 $\theta\in C^\infty[0,\infty)$ is a nondecreasing cut-off function such that 
\be{localtheorycut1}
\theta(s)=
\begin{cases}
s\quad\hbox{ if }\quad 0\le s\le 1/2,\\
1\quad\hbox{ if }\quad s\ge 1,
\end{cases}\hspace{1cm} \hbox{and }\quad
s\theta'(s)+s^2|\theta''(s)|\le C_2\theta,
\ee  
and  
\be{localtheorymu}
\mu=C(p)M(3+\beta+C_2)C_1^{p-1}>0\quad \hbox{with some}\ C(p)>0.
\ee
 Also assume $\tau\le \mu^{-p}$. We have 
$h_\eps(x_0,t_0) \le \eps A$, where 
$$
A:=\Big[e M (1+x_{0,n})^{2+\beta} +\log(1+|x_0'|^2)\Big].
$$
Choose $\eps=\min(1/2, A^{-1}\tilde\sigma)$. By \eqref{localtheory3.1''}, we have 
\be{localtheory3.3'}
Z_\eps(x_0,t_0)\ge 2\tilde\sigma.
\ee
 Using $v\ge 0$, $u\in L_w^\infty(Q_T)$ and the first part of \eqref{localtheorycut1}, we note that there exists $R_\eps>R_0$ such that  
\be{localtheory3.3''}
Z_\eps(x,t)<0 \quad\hbox{in }\ \Big(\{ |x|\ge R_\eps \}\cap H\Big)\times(0,\tau]
\ee
 (separating the cases $x_n>R_0+\eps^{-1}$ and $x_n\le R_0+\eps^{-1}$ with $|x'|$ large enough).

Now set $Q'=B_{R_\eps}\cap\{x_n>R_0\})\times(0,\tau]$. In view of \eqref{localtheory3.1'}, \eqref{localtheory3.3'}, \eqref{localtheory3.3''} and $u(\cdot,0)\le v(\cdot,0)$, the function $Z_\eps$ achieves its maximum in the compact $\bar Q'$ at some $(x_1,t_1)\in Q'$.
 We will see that this leads to a contradiction.
At this point we have 
\be{localtheorytronc1}
|\nabla Z_\eps(x_1,t_1)|=0,\quad LZ_\eps (x_1,t_1)\ge 0,
\ee
hence in particular
\be{localtheorytronc2}
|\nabla (u-v) (x_1,t_1)|\le  |\nabla h_{\eps}(x_1,t_1)|\quad \hbox{and }\ |\nabla v (x_1,t_1)|\le |\nabla u(x_1,t_1)|+ |\nabla h_{\eps}(x_1,t_1)|.
\ee
Since $u$ and $v$ are classical sub/super-solutions of \eqref{localtheoryeqE1b}, we have 
$$
\begin{aligned}
L Z_\eps &\le|\nabla u|^p-|\nabla v|^p- L h_\eps
&\le |B(x,t)|| \nabla (u-v)(x,t)|- L h_\eps,
\end{aligned}
$$
where 
$$
B(x,t):=\int_0^1 G\Big(s\nabla v(x,t)+(1-s)\nabla u(x,t)\Big)ds,\quad G(\xi)=p|\xi|^{p-2}\xi.
$$
Therefore, combining with \eqref{localtheorygradientu} and \eqref{localtheorytronc2} at the point $(x_1,t_1)\in Q'$, we obtain 
$$
\begin{aligned}
LZ_\eps(x_1,t_1)&\le C_3\Big(|\nabla u(x_1,t_1)|^{p-1}+|\nabla h_\eps(x_1,t_1)|^{p-1}\Big) |\nabla h_\eps(x_1,t_1)|-L h_\eps(x_1,t_1)\\
&\le C_3 \Big( C_1^{p-1}(1+x_{1,n})t_1^{-\frac{p-1}{p}}|\nabla h_\eps(x_1,t_1)|+|\nabla h_\eps(x_1,t_1)|^{p}\Big) -L h_\eps(x_1,t_1),\\
\end{aligned}
$$
where $C_3:=C_3(p)$ is a positive constant.
At this stage,
it suffices to prove that the RHS is a negative function 
by a suitable choice of $C(p)>0$ in \eqref{localtheorymu}. By direct computation, we obtain
$$
|\nabla h_\eps(x,t)|\le 2\eps+ M(1+x_n)^{\beta}\Big((1+\beta)\theta +\eps(1+x_n) \theta'\Big)e^{\mu t^{1/p}},\quad \partial_th=\frac{\mu t^{(1-p)/p}}{p}(1+x_n)^{\beta+1}\theta e^{\mu t^{1/p}}
$$
and 
$$
\Delta h\le 2n\eps+ M(1+x_n)^{\beta-1}\Big((1+\beta)\beta\theta +2\eps(1+x_n) |\theta'|+\eps^2(1+x_n)^2 |\theta''|\Big)e^{\mu t^{1/p}},
$$
where the argument of $\theta,\theta'$ and $\theta''$ is omitted, i.e. $\theta:=\theta(\eps(1+x_n))$ as well as $\theta'$ and $\theta''$. Together with \eqref{localtheorycut1}, which in particular implies
\be{localtheoryA}
\theta(\eps(1+x_n))\ge \theta(\eps)\ge \eps,
\ee
 we arrive at
$$
|\nabla h_\eps(x,t)|\le  M(1+x_n)^{\beta}\Big(3+\beta +C_2 \Big)\theta(\eps (1+x_n))e^{\mu t^{1/p}}=:C_4(1+x_n)^{\beta}\theta(\eps (1+x_n))e^{\mu t^{1/p}},
$$
and 
$$
\Delta h\le 2n\eps+ M(1+x_n)^{\beta-1}\Big((1+\beta)\beta +2C_2\Big) \theta(\eps (1+x_n))e^{\mu t^{1/p}}=:2n\eps+C_5(1+x_n)^{\beta-1} \theta(\eps (1+x_n))e^{\mu t^{1/p}}.
$$
Thus choosing $C(p)=2pC_3$ in \eqref{localtheorymu}, we get
$$
\begin{aligned}
&C_3 \Big( C_1^{p-1}(1+x_{1,n})t_1^{-\frac{p-1}{p}}|\nabla h_\eps(x_1,t_1)|+|\nabla h_\eps(x_1,t_1)|^{p}\Big) -L h_\eps(x_1,t_1)\\
&\le (1+x_{1,n})^{\beta+1}\theta \Bigg\{C_3\Big( C_1^{p-1}t_1^{-\frac{p-1}{p}}C_4e^{\mu t_1^{\frac{1}{p}}}+C_4^p e^{p\mu t_1^{1/p}}\Big)-\frac{\mu}{p}t_1^{\frac{1-p}{p}}e^{\mu t_1^{\frac{1}{p}}} + C_5e^{\mu t_1^{\frac{1}{p}}}\Bigg\}+2n\eps\\
 &\le (1+x_{1,n})^{\beta+1}\theta e^{\mu t_1^{\frac{1}{p}}}\Bigg\{\Big( C_3 C_1^{p-1}C_4-\frac{\mu}{p}\Big) t_1^{-\frac{p-1}{p}}+C_3C_4^pe^{(p-1)\mu t_1^{1/p}}+ C_5\Bigg\}+2n\eps\\
&\le \theta(\eps(1+x_{1,n}))(1+x_{1,n})^{\beta+1}e^{\mu t_1^{\frac{1}{p}}}\Bigg\{-\frac{\mu}{2p}\tau^{-\frac{p-1}{p}}+C_3C_4^pe^{(p-1)}+ C_5\Bigg\}+2n\eps.\\
\end{aligned}
$$
 Next choosing $\tau\in(0,\mu^{-p}]$ sufficiently small such that
$$
-\frac{\mu}{2p}\tau^{-\frac{p-1}{p}}+C_3C_4^pe^{(p-1)}+ C_5\le -(2n+1),
$$
and using \eqref{localtheoryA}, we obtain $L Z_\eps(x_1,t_1)\le -(2n+1)\theta+2n\eps\le -\eps<0$,
which contradicts \eqref{localtheorytronc1}. 
Hence, there is a small $\tau>0$ such that $u\le v$ in $H\times[0,\tau]$

Set 
$$
\tilde{T}:=\sup\{\tau\in (0,T):\ u\le v\ \hbox{ in }  H\times[0,\tau]\}.
$$
 If $\tilde{T}<T$, we observe that the assumptions of Proposition~\ref{localtheoryuniqueness} are satisfied on $[\tilde{T},T)$ with $u(\cdot,\tilde{T})\le v(\cdot,\tilde{T})$. 
 By the above argument applied to $u(\cdot,\tilde{T}+t),\  v(\cdot,\tilde{T}+t)$,
 there exists $\tau>0$ such that $u\le v$ on $ H\times[\tilde{T},\tilde{T}+\tau]$. This contradicts the definition of $\tilde{T}$. Hence $\tilde{T}=T$, and Proposition~\ref{localtheoryuniqueness} is proved.
\end{proof}

\begin{proof}[Proof of Proposition~\ref{localtheoryuni}] Let $u_0\in X$. Let $\tau>0$ and $u_1,u_2\in Y_\tau$ be two solutions of \eqref{localtheoryeqE1b} with the same initial data $u_0$.
By Proposition~\ref{localtheoryuniqueness}, we directly obtain that $u_1= u_2$ in $ H\times(0,\tau)$. This concludes the proof. 
\end{proof}

\subsection{Nonexistence: proof of Theorem~\ref{localtheorynonexistence}}\label{localtheoryNE}

 Assume that there exists a classical solution $u\in C^{2,1}(Q_T)\cap C(\bar Q_T) $ of \eqref{localtheoryeqE1b} defined on $ H\times(0,T)$ for some $T>0$ with initial data $u_0$ as in the statement of the Theorem~\ref{localtheorynonexistence}.
 Set $\tau=T/2$, we now consider the self-similar solution from \cite[Theorem~4]{chso2}, given by
$$ 
v(x,t)=({  \tau}-t)^\gamma\phi(x_n/\sqrt{\tau-t})\quad\hbox{in }\  H\times(-\infty,\tau),
$$ 
where $\gamma=\frac{p-2}{2(p-1)}$ and the profile $\phi\in C^2([0,\infty))$ satisfies $\phi(0)=0$, $\phi'>0$ on $[0,\infty)$ and 
 $\displaystyle\lim_{y\to\infty} \frac{\phi(y)}{y^{\beta+1}}=L>0$.
 There exists $C:=C(\tau,\phi)>0$ such that 
 $$
v_0(x):=v(x,0)\le C(1+x_n)^{\beta+1},\quad \hbox{for }\ x\in  H.
$$ 
 In particular, there exists $R_0:=R_0(C)>0$ such that
 \be{localtheory7'}
u_{0}(x+R_0 e_n)\ge 2C (1+x_n)^{\beta+1}\ge v_0(x) \quad \hbox{in }\quad H.
 \ee
 Setting $\tilde{u}(x,t)=u(x+R_0e_n,t)$ for $(x,t)\in  H\times[0,T)$ and using $v\in Y_T$,
 Proposition~\ref{localtheoryuniqueness} yields $v\le \tilde u$ in $ \bar Q_\tau$.
 Since, for each $x\in  H$, $v(x,t)\to \infty$ as $t\to \tau$, this is a contradiction with the fact that $v\in C( \bar Q_\tau)$.

\qed

\subsection{Local existence part of Theorems~\ref{localtheorymaxsol}(i) and \ref{localtheorymaxsol1p2} for $p>1$}\label{localtheoryExistencesection}
 Let note that $u_0\in X$ implies that there exist $M_0\ge1$ such that
\be{localtheoryassumptu0}
0\le u_0(x)\le M_0(1+x_n)^{1+\beta}\quad \hbox{in } H
\ee
and, in addition, if  $u_0\in X_1$ then \eqref{localtheoryassumptu0} holds and there exist $0<M_1<c_p$ and $a\in(0,1)$ such that
\be{localtheoryassumptu01}
 u_0(x)\le M_1x_n^{1-\beta}\quad \hbox{in }\ \Gamma_a.
\ee
 Throughout this section, $C$ will denote a generic positive constant depending only on $n$ and $p$, unless otherwise specified.

\smallskip

\noindent{\bf Step 1. Regularization process.} 
 Let $\rho_k=\rho_k(x')$ be a standard mollifying sequence in the variables $x'\in\R^{n-1}$,
extend $u_0$ by $0$ for $x_n<0$, and then set
$$
u_{0,k}(x)=k\int_{x_n-k^{-1}}^{x_n}\int_{\R^{n-1}} \min(k,u_0(x'-y',y_n)) \rho_k(y') dy' dy_n.
$$
Then the  sequence 
$(u_{0,k})_k\subset BC^1_0(\bar  H)$ satisfies \eqref{localtheoryassumptu01} for every $u_{0,k}$, and 
\be{localtheory3.9'}
0\le u_{0,k}(x)\le 2 M_0(1+x_n)^{1+\beta},
\ee
and $u_{0,k}\to u_0$ as $k\to \infty$ locally uniformly on $\bar H$. 
 Now consider the problem of solving the equation
\begin{equation}\label{localtheoryeqE1C}
	\begin{cases}
	Lu=|\nabla u|^p,&x\in H,\ t>0,\\
	u=0,&x\in \partial H,\ t>0,\\
	u(0,x)=u_{0,k}(x),&x\in  H,
	\end{cases}
\end{equation} 
 Since $u_{0,k}\in BC_{0}^1(\bar H)$, this problem has a nonnegative classical solution $u_k\in C^{2,1}( H\times (0,T_k))\cap BC^1(\bar H\times [0,T_k))$ for some $T_k>0$ and $u_k\in BC^{2,1}(\bar H\times [t_1,t_2])$  for any $0<t_1<t_2<T_k$, cf. \cite[Section~35]{quittner2019superlinear}.
 Moreover, we have 
\be{localtheorystar1}
u_{k}(x,t)\le \|u_{0,k}\|_{L^\infty( H)}<\infty\quad \hbox{in } H\times[0, T_k).
\ee
 Indeed, since any constant is a solution of \eqref{localtheoryeqE0} and 
 \be{localtheory3.11'}
 (C^{2,1}(Q_\tau)\cap BC^1(\bar Q_\tau))\subset Y_\tau\quad \hbox{for } \ \tau>0,
 \ee
 the comparison principle Proposition~\ref{localtheoryuniqueness} 
 yields \eqref{localtheorystar1}.

 \smallskip

\noindent{\bf Step 2. Control of $u_k$ in $\R^n_+$}. 
Let 
\be{localtheorycondition1-1}
 T_0= \alpha^{p-1} M_0^{1-p}\quad\hbox{where }\ 
\alpha=\alpha(p)=\min\Bigg\{\Big[\frac{\beta}{2}(\beta+1)^{-p}\Big]^\beta,\Big(\frac{p-1}{2p}\Big)^\beta\Bigg\}.
\ee
We claim that
\be{localtheoryCC1}
u_k(x,t)\le 2^{1+\beta} M_0(1+x_n)^{1+\beta}\quad \hbox{ in }\quad  \R^n_+\times (0,\tau_k],\quad \hbox{with }\ \tau_k=\min\{T_k,T_0/2\}.
\ee

Set
\be{localtheorydefv1}
v_1(x,t):=\alpha(T_0- t)^{-\beta}(1+x_n)^{1+\beta}.
\ee
A direct computation yields
$$
Lv_1-|\nabla v_1|^p=\Bigg( (T_0-t)^{-1}\Big[\beta-\alpha^{p-1}(\beta+1)^p\Big]-p\beta^2(1+x_n)^{-2} \Bigg) v_1\ge\Big(\frac{(T_0)^{-1}}{2}-p\beta\Big)\beta v_1. 
$$
 Since $T_0\le (2p\beta)^{-1}$, this
 implies that 
$$
Lv_1\ge |\nabla v_1|^p,\ \  \hbox{for all $(x,t)\in Q_{T_0}$},
$$ 
and for all $x\in  H$
$$
v_1(x,0)=\alpha T_0^{-\beta}(1+x_n)^{\beta+1}\ge 2 M_0(1+x_n)^{\beta+1}\ge u_{0,k}(x).
$$
 By Proposition~\ref{localtheoryuniqueness} and \eqref{localtheory3.11'}, 
we deduce that 
$$
u_k(x,t)\le v_1(x,t),\quad x\in  H,\ 0<t<\min\{T_k,T_0\}.
$$
 This implies \eqref{localtheoryCC1}.

\smallskip

 \noindent{\bf Step 3. Refined control of $u_k$ in $\Gamma_a$ for $p>2$.} 
 By \eqref{localtheory3.9'} we have
$$
u_{0,k}(x)\le M_0 2^{2+\beta},\quad x\in \bar \Gamma_a.
$$
We shall prove that
\be{localtheoryCCC1}
u_k\le v_2\quad \hbox{in }\  \Gamma_a\times (0,\tau_k),
\ee
where $v_2$ will be a supersolution defined as
$$
v_2(x,t):=  C_0(p) M_0^p a^{-p}t+M_2 g(x_n), \quad (x,t)\in \Gamma_a\times(0,\tau_k),
$$
 where 
$$
M_2:=(M_1+c_p)/2\quad\hbox{ and }\quad C_0(p)>0 \ \hbox{to be chosen later},
$$
 and $g$ is the following piecewise function:
$$
g(s)=
\begin{cases}
\ s^{1-\beta},\ \hspace{3.6cm} \hbox{ for } s\le \rho_1\\
\ \rho_1^{1-\beta}+(1-\beta)\rho_1^{-\beta}(s-\rho_1),\ \ \hbox{ for } \rho_1< s\le a,
\end{cases}
$$ 
with $\rho_1\in (0,a/2)$  given by 
\be{localtheoryrho1.2}
\rho_1= C_1(p) M_0^{1-p}a^{p-1}.
\ee
We first note that $v_2(\cdot,t)\in C^1(\Gamma_a)$ and  $v_2(\cdot,t)\in C^2(\bar\Gamma_{\rho_1})\cap C^2(\bar\Gamma_{\rho_1,a})$, therefore $v_2(\cdot,t)\in H^2_{loc}(\bar \Gamma_a)$.

 On the other hand, 
  we have $g(s)\ge s^{1-\beta}$ on $[0,a]$, owing to $\frac{d}{ds}(g(s)-s^{1-\beta})\ge 0$ on $(\rho_1,a]$. 
 This along with
 \eqref{localtheoryassumptu01}
guarantees that 
$$
v_2(x,0)=  M_2 g(x_n)\ge   u_{0,k}(x),\quad \hbox{for all } x\in \bar\Gamma_{a}.
$$
Also since $u_k\in BC^1(\bar H\times (0,\tau_k))$ and $v_2\ge 0$,
 for any $\eps>0$, there exists $\eta_k>0$ such that 
$$
u_k-v_2\le \eps \quad \hbox{in }\ \bar\Gamma_{\eta_k}\times(0,\tau_k).
$$
Thus it remains to prove that, for any $\eps>0$ and its corresponding $\eta_k$, 
$$
u_k-v\le \eps \quad \hbox{in }\ \Gamma_{\eta_k,a}\times(0,\tau_k).
$$
 By \eqref{localtheoryrho1.2} with sufficient small $C_1(p)>0$, we have $ M_2g(a)\ge M_2 (1-\beta)\rho_1^{-\beta}a/2\ge c_p/2(1-\beta)C_1^{-\beta}(p) M_0/2\ge 2^{2+2\beta}M_0$. It then follows from \eqref{localtheoryCC1} that 
\be{localtheoryutil}
v_2(x',a,t)\ge M_2 g(a)\ge  u_k(x',a,t),\quad \hbox{for all }\  t\in (0,\tau_k),\ x'\in \R^{n-1}.
\ee
Moreover, a simple computation yields:
$$
\partial_t v_2(x,t)=C_0(p) M_0^pa^{-p},\quad |\nabla v_2(x,t)|= \begin{cases}
M_2(1-\beta) x_n^{-\beta},\  \hbox{ for } x_n\le \rho_1\\
M_2(1-\beta)\rho_1^{-\beta} \ \hbox{ for } \rho_1\le s\le a,
\end{cases}
$$
and 
$$
\Delta v_2(x,t)= \begin{cases}
-M_2\beta (1-\beta) x_n^{-\beta-1},\  \hbox{ for } x_n< \rho_1\\
0 \hspace{2cm} \hbox{ for } \rho_1< s\le a.
\end{cases}
$$
 Now choosing $C_0(p)=d_p^pC_1^{-p/(p-1)}$ and using $M_2< c_p$, we get that $L v_2\ge |\nabla v_2|^p$ in the weak sense.
Using \eqref{localtheoryutil}, the comparison principle \cite[Proposition 52.10]{quittner2019superlinear} (cf. Proposition~\ref{localtheory52.10} below) implies
$$
u_k(x,t)-v_2 (x,t)\le \eps\quad\hbox{for all } (x,t)\in \Gamma{\eta_k,a}\times (0,\tau_k).
$$
We already proved that, for any $\eps>0$,
$$
 u_k(x,t)-v_2 (x,t)\le \eps\quad \hbox{for all } (x,t)\in \Gamma{\eta_k}\times (0,\tau_k).
$$
We conclude the proof of \eqref{localtheoryCCC1} by letting $\eps\to 0$.

\smallskip

\noindent{\bf Step 4. Control of $\nabla u_k$ at $\partial H$  for $p>2$}.
 We shall prove that 
\be{localtheorycontrol1-1}
 |\nabla u_k(x',0,t)|\le d_pt^{-1/(p-2)} \quad\hbox{for all } (x',t)\in \R^{n-1}\times(0,\tilde{T}_k),
\ee
  where $\tilde{T}_k$ is defined by 
  \be{localtheorydeftildeT}
 \tilde{T}_k= \min\Big\{T_k,T_0/2,C(p)(c_p-M_1)^\frac{2(p-1)}{p}(a/M_0)^{2(p-1)}\Big\},
\ee
with some $C(p)>0$.

 Recall \eqref{localtheoryrho1.2} and let
\be{localtheoryrho}
\rho:=\min\Big\{\rho_1,C(p)(c_p-M_1)^{\frac{p-1}{p}} \Big\}=\min\Big\{C_1(p)(a/M_0)^{p-1},C(p)(c_p-M_1)^{\frac{p-1}{p}} \Big\}.
\ee
In order to prove  estimate \eqref{localtheorycontrol1-1}, we are going to construct a local barrier function, 
  namely, a supersolution 
 in $\Gamma_{\rho}\times (0,\tilde T_k)$ vanishing at $x_n=0$, modifying a device from \cite{PS_imrn16}.
   We look for a barrier function of the form   
     \be{localtheoryvopt}
  v(x,t)=c_p\Big(\left(x_n+t^{1/(1-\beta)}\right)^{1-\beta}-t-x_n^2/2\Big),\quad\hbox{ for $0\le x_n\le\rho$ and } t\in(0,\tilde T_k).
  \ee
 We compute 
 $$
 \partial_t v(x,t)=c_p\left(1+x_nt^{\frac{-1}{1-\beta}}\right)^{-\beta}-c_p,\quad \partial_{x_n}v(x,t)=c_p(1-\beta)\left(x_n+t^{\frac{1}{1-\beta}}\right)^{-\beta}-c_p x_n
 $$
 and 
 $$
 -\Delta v=\beta c_p(1-\beta)\left(x_n+t^{\frac{1}{1-\beta}}\right)^{-\beta-1}+c_p.
 $$
   Since $\rho\le (1-\beta)/2$ for $C(p)$ small in \eqref{localtheoryrho}, a simple calculation yields 
$$
v(x,t)\ge 0,\quad\partial_{x_n} v\ge 0,\quad\hbox{hence}\quad |\nabla v|^p\le \Big(c_p(1-\beta)\Big)^p\left(x_n+t^{\frac{1}{1-\beta}}\right)^{-p\beta}.
$$
 Using the definition of $c_p$ and $p\beta=\beta+1$, we have  for all $(x,t)\in \Gamma_\rho\times(0,\tilde{T}_k)$:
\be{localtheoryedp1}
\begin{aligned}
L  v=& c_p\left(1+x_nt^{\frac{-1}{1-\beta}}\right)^{-\beta}+\beta c_p(1-\beta)\left(x_n+t^{\frac{1}{1-\beta}}\right)^{-\beta-1}\\
\ge & \beta c_p(1-\beta)\left(x_n+t^{\frac{1}{1-\beta}}\right)^{-\beta-1}\ge |\nabla v|^p.
\end{aligned}
\ee
Also by using \eqref{localtheoryassumptu01}, \eqref{localtheoryrho} with sufficiently small $C(p)$,  we obtain, for all $x\in \Gamma_\rho$,
$$
v(x,0)=c_p(x_n^{1-\beta}-x_n^2/2)\ge u_{0,k}(x).
$$
We shall apply Proposition~\ref{localtheoryuniqueness} together with Remark~\ref{localtheoryuni+1}(ii) in $\Gamma_{\rho}\times(0,\tilde T_k)$. To this end, we also need the following control  
$$
u_k(x',\rho,t)\le v(x',\rho,t)\quad \hbox{ for all}\ t\in (0,\tilde T_k),\quad x'\in \R^{n-1}.
$$
By the conclusion \eqref{localtheoryCCC1} of Step 3 (where $\tau_k$ is defined in \eqref{localtheoryCC1}) and the fact that $\rho\le \rho_1$, it suffices 
to have 
$$
v(x',\rho,t)\ge v_2(x',\rho,t)=C_0(p)M_0^pa^{-p}t+M_2 \rho^{1-\beta} \quad  \hbox{in }\  (0,\tilde T_k).
$$
 For sufficiently small $C(p)>0$ in \eqref{localtheorydeftildeT}, which in particular implies $\tilde{T}_k\le C(p) (a/M_0)^p(c_p-M_1)\rho^{1-\beta} $, a simple calculation shows that
$$
v(x',\rho,t)\ge {  c_p(\rho^{1-\beta}-\tilde T_k-\rho^2/2)\ge v_2(x',\rho,\tilde T_k)\ge v_2(x',\rho,t).}
$$
Therefore, by \eqref{localtheoryedp1} and Proposition~\ref{localtheoryuniqueness}, we obtain
\be{localtheoryconcl1}
u_k(x,t)\le v(x,t)\quad\hbox{in }\ \Gamma_{\rho}\times(0,\tilde T_k).
\ee
A direct consequence of \eqref{localtheoryconcl1} is that 
$$
|\partial_{x_n} u_k(x',0,t)|\le \partial_{x_n} v(x',0,t)\le d_p t^{-1/(p-2)}\quad \hbox{ for  all } x'\in \R^{n-1},\ t\in (0,\tilde T_k).
$$
Moreover, since $u_k= 0$ on $\partial H$, this implies \eqref{localtheorycontrol1-1}.

\smallskip
\noindent {\bf Step 5. Control of  $\nabla u_k$ at $\partial H$  in the case $1<p\le2$}.
 We claim that 
\be{localtheoryCC1p12}
|\nabla u_k(x',0,t)|\le C(n,p,M_0) t^{-1/2},\quad\hbox{ for }\ 0<t<\tilde{T}_k:=\min\{T_0/2,T_k\}, \ x'\in \R^{n-1}.
\ee
As in the previous step, we proceed by the barrier approach.
Here we define 
$$
v(x,t)=\log(1+U(x_n,t))\quad \hbox{in}\quad \Gamma_1\times(0,1),
$$
where  $U(y,t)$ is the solution of the one dimensional problem
$$
U_t-U_{yy}=1, \quad U(y,0)=K_1:= \exp (2^{2(\beta+1)}M_0)-1,\quad U(0,t)=0,\quad U(1,t)=K_1.
$$
  Note that $\tau_k\le 1$ by \eqref{localtheorycondition1-1}, \eqref{localtheoryCC1}.
By a simple computation, we see that $v_t-\Delta v\ge |\nabla v|^p$ in  $\Gamma_1\times(0,1)$ and 
 $v\ge u_k$ on the parabolic boundary of $ \Gamma_1\times(0,\tau_k)$ owing to \eqref{localtheory3.9'}, \eqref{localtheoryCC1}. The standard comparison principle yields $u_k\le v $ on $ \Gamma_1\times(0,\tau_k)$.

 Set $h(y)=K_1y$ and $V(y,t)=U(y,t)-h(y)$. We have
$$
V_t-V_{yy}=1, \quad V(y,0)=K_1(1-y),\quad U(0,t)=0,\quad U(1,t)=0,
$$
and by the variation of constants formula (Duhamel formula), we get
$$
V(y,t)=  e^{tA}V(y,0)+\int_0^t e^{(t-s)A}1 ds
$$
 where $A$ is the 1d Dirichlet Laplacian on  $(0,1)$. Hence by using $V(y,0)\le K_1$ we arrive at
$$
U(y,t)\le K_1 e^{tA}1+K_1y+\int_0^t e^{(t-s)A}1 ds.
$$ 
 Therefore, by standard estimates for the heat equation (see, e.g., \cite[Proposition~49.9]{quittner2019superlinear}), and using $\log(1+s)\le s$ for $s\ge 0$, we obtain 
\be{localtheoryfinal2}
u_k(x,t)\le v(x,t)\le U(x_n,t)\le C\Big(K_1t^{-1/2}+K_1+t^{1/2})\Big)x_n\quad \hbox{in}\quad \Gamma_1\times(0,1).
\ee
 Recalling \eqref{localtheoryCC1}, we then get that 
\be{localtheory3.x}
|\partial_{x_n}u_k(x',0,t)|\le C\Big(K_1t^{-1/2}+K_1+t^{1/2})\Big)\le C(n,p,M_0) t^{-1/2},\quad\hbox{ for }\ 0<t<\tilde{T}_k,
\ee
hence \eqref{localtheoryCC1p12}.

\noindent{\bf Step 6. Gradient estimate in $\Gamma_1$ and conclusion.} We shall prove that 

\be{localtheoryCC2}
|\nabla u_k(x,t)|\le C(n,p,M_0)(1+t^{-\gamma})\quad \hbox{in }\ \bar \Gamma_1\times (0,\tilde T_k),
\ee
 where $\gamma=1/(p-2)$ (resp.  $1/p$) and $\tilde T_k$ is defined by \eqref{localtheorydeftildeT} (resp. \eqref{localtheory3.x}) for $p> 2$ (resp. $p\in(1,2]$). 
Using $u_k\ge 0$, \eqref{localtheoryCC1}, \eqref{localtheorycontrol1-1} for $p>2$ and \eqref{localtheoryCC1p12} in the case $1<p\le 2$, 
we derive \eqref{localtheoryCC2} as consequence of the following property.

\begin{lem}\label{localtheoryBernchso}
Let $p>1$ and $u\ge0$ be a  classical solution of \eqref{localtheoryeqE0} in $ Q:=(B_1\cap\R^n_+)\times (0, T)$, which satisfies 
\be{localtheoryhypoth}
u\le M\ \hbox{in }\ Q ,\quad |\nabla u|\le \tilde M t^{-\gamma_0} \ \ \hbox{ and $\ u=0$ on }\ Q\cap \{x_n=0\},
\ee
 for some $\gamma_0, M,\, \tilde{M}>0$. Then there exists $C_0:=C_0(n,p,M,\tilde M)>0$ such that
 $$
|\nabla u|\le C_0(1+ t^{-\bar\gamma}), 
\quad \hbox{in } (B_{1/2}\cap\overline{\R^n_+})\times (0 , T)
$$
 with $\bar\gamma=\max(\gamma_0,1/p)$.
\end{lem}
The proof of this Bernstein type property is similar to that of \cite[Theorem~8]{chso2}. A short proof is provided in Appendix~\ref{localtheorylemma3.1} for convenience of readers.

Therefore, by \eqref{localtheorystar1}, \eqref{localtheoryCC2} in $\Gamma_1$ and the Bernstein type estimate \eqref{localtheorygrabound} for $x_n\ge 1/2$, we have
\be{localtheoryk_0}
\|u_{k}(\cdot,t)\|_{BC^1( H)}\le \|u_{0,k}\|_{L^\infty( H)}+C  (1+ t^{-\gamma}),\quad 0< t <
\tilde{T}_k.
\ee
Since $T_k<\infty$ implies 
$$
\|u_{k}(\cdot,t)\|_{BC^1( H)}\to \infty, \quad \hbox{as }\ t\to T_{k},
$$
then \eqref{localtheoryk_0} immediately implies $T_k\ge T_0/2$ for $1<p\le 2$ and 
$$
T_k\ge \min\{T_0/2,C(p)(c_p-M_1)^\frac{2(p-1)}{p}(a/M_0)^{2(p-1)}\}
$$
 in the case of $p>2$. 
  Recall \eqref{localtheorycondition1-1}, for $C(p)$ small enough, we have 
 $$
 C(p)(c_p-M_1)^\frac{2(p-1)}{p}(a/M_0)^{2(p-1)}\le \frac{\alpha^{p-1}M_0^{1-p}}{2}= T_0/2
 $$
  owing to
$M_0\ge 1$ and $a\le 1$, we then deduce that
\be{localtheorytempscourt}
\tilde{T}_k\ge  T:=\begin{cases}
T(p,M_0)=C(p)M_0^{1-p}\quad \hspace{4cm}\hbox{if }\ 1<p\le2,\\
\vspace{2pt}
 T(p,a,M_0,M_1)=C(p)(c_p-M_1)^\frac{2(p-1)}{p}(a/M_0)^{2(p-1)}\quad \hbox{if }\ p>2.
\end{cases}
\ee  

 Now, by \eqref{localtheoryCC1}, the uniform gradient bounds \eqref{localtheoryCC2} for $x_n<1$ and  \eqref{localtheorygrabound} for large $x_n$, together with  interior and boundary parabolic estimates and standard embeddings, the sequence $(u_k)_k$ is relatively compact in $C^{2,1}_{loc}( Q_T)$. Hence, there exists a subsequence of $(u_k)_k$, which converges to  $u$ as $k\to \infty$ in $C^{2,1}_{loc}(Q_T)$. Moreover, $u\in C^{2,1}( Q_T)$, $u\ge0$ and u solves \eqref{localtheoryeqE1b} in the classical sense. 
By passing to the limit $k\to \infty$ in the uniform bound \eqref{localtheoryCC1}, we obtain 
$$
(1+x_n)^{-1-\beta}u \in L^{\infty}( H\times (0,T)).
$$
In addition, using the latter together with \eqref{localtheoryCC2} and \eqref{localtheorygrabound}, we conclude that 
\be{localtheoryconclu2}
|\nabla u(x,t)|\le C(n,p,M_0)(1+x_n)^\beta (1+ t^{-\gamma})\quad\hbox{in }\  H\times(0,T).
\ee
To close the proof of the local existence in the class $Y_T$, it subsequently remains to show that $u$ can be extended as $t\to0^+$ in such a way that 
\be{localtheorycontinuit}
u\in C(\bar  H\times[0,T))\quad\hbox{ and }\  u(x,0)=u_0(x).
\ee
 First, 
by letting $k\to \infty$ in \eqref{localtheoryfinal2} for $p\in(1,2]$ or in \eqref{localtheoryconcl1} for $p>2$ and using \eqref{localtheoryvopt}, we obtain $u\in C(\bar  H\times (0,T])$ and 
\be{localtheoryfinal1}
\sup_{\Gamma_r\times(0,T)}u\to 0\quad\hbox{as }\ r\to 0.
\ee
To check the continuity as $t\to 0$, we 
let $x_0\in H$, $t_0\in (0,T)$, $\eps>0$ and $\delta_0=x_{0,n}/2$.
We already know that 
\be{localtheory3.32'}
u_k\le C(n,p,x_{0,n},t_0)=C \quad \hbox{in }\  Q_{\delta_0} =B_{\delta_0}(x_0)\times [0,t_0].
\ee
 For suitable $\delta\in (0,\delta_0)$, we shall compare $u_k$ with the function
$$
W_1(x,t)=u_0(x_0)+\eps+d_1 t+\frac{C}{\delta^2} |x-x_0|^2\quad \hbox{with $d_1=\frac{2C}{\delta^2}+\Big(\frac{2C}{\delta}\Big)^p$ in }\ Q_\delta.
$$
Using the continuity of $u_0$ and the local uniform convergence of $(u_{0,k})_k$ to $u_0$ on $\bar H$, there exist $\delta\in (0,\delta_0)$ and $k_0:=k_0(\delta)>0$ such that 
$$
u_{0,k}(x)\le u_0(x_0)+\eps\quad\hbox{in }\ B_\delta(x_0)
$$
for any $k\ge k_0$. Therefore, we have $W_1(x,0)\ge u_{0,k}(x)$ in $B_\delta(x_0)$ and 
by definition of $W_1$ and \eqref{localtheory3.32'} we have $W_1\ge u_k$ in $ (\partial B_\delta(x_0))\times (0,t_0)$.
Moreover, 
a simple calculation yields
$$
\begin{aligned}
L W_1= d_1-\frac{2C}{\delta^2}=\Big(\frac{2C}{\delta}\Big)^p\ge|\nabla w_1|^p\quad \hbox{in } \ Q_\delta.
\end{aligned}
$$
By the standard comparison principle we have 
$$
u_k(x,t)\le W_1(x,t)\quad\hbox{in }\ Q_\delta.
$$
By the same way, we obtain 
$$
u_0(x_0)-\eps-d_2 t-\frac{C}{\delta^2} |x-x_0|^2\le u_k(x,t)\quad \hbox{in }\ Q_\delta.
$$
 Letting $k\to \infty$, we
in particular get
$$
u_0(x_0)-\eps-d_2 t\le u(x_0,t)\le u_0(x_0)+\eps+d_1 t\quad \hbox{in }\ [0,t_0].
$$
Since $\eps>0$ is arbitrary, we may let $\eps\to 0$ and then $t\to 0$. This combined with \eqref{localtheoryfinal1} implies \eqref{localtheorycontinuit}.

\subsection{Maximal solution: completion of proof of Theorems~\ref{localtheorymaxsol} and \ref{localtheorymaxsol1p2}}\label{localtheorysectionmax}
\noindent $\bullet$ {\bf Proof of Theorem~\ref{localtheorymaxsol} (i)-(ii) and Theorem~\ref{localtheorymaxsol1p2}(i)-(ii).} Let $p>2$, $u_0\in X_1$ and assume that $u_1$ and $u_2$ are two solutions of \eqref{localtheoryeqE1b} on $[0,T_1)$ and $[0,T_2)$, respectively. Since $X_1\subset X$, then $u_1=u_2$ on $[0,\min\{T_1,T_2\})$ due to the uniqueness (Proposition~\ref{localtheoryuniqueness}).
Now we define $\{u_\alpha: [0,T_\alpha)\to X\}$, the set of all solutions of \eqref{localtheoryeqE1b}, which is nonempty by the local existence part proved in Section~\ref{localtheoryExistencesection}, and $\tau=\sup T_\alpha$. We also define $u:[0,\tau)\to X$ by $u(t)=u_\alpha(t)$, where $\alpha$ is any index such that $T_\alpha>t$. Then $u$ is
obviously a solution of \eqref{localtheoryeqE1b} on $[0,\tau)$ and, together with \eqref{localtheoryconclu2}, properties $(i)$ and $(ii)$ of the Theorem~\ref{localtheorymaxsol} are verified. Arguing similarly, we obtain the first assertion of Theorem~\ref{localtheorymaxsol1p2}.

\smallskip

\noindent $\bullet$ {\bf Proof of blow-up alternative for $p>1$}.
 Assume by contradiction that $\tau<\infty$ and that
$$
\begin{cases}
\liminf_{t\to \tau} \|u(\cdot,t)\|_{L_w^\infty( H)}<\infty\quad\hbox{if }\ 1< p\le2\\
\vspace{2pt}
 \liminf_{t\to \tau} \|u(\cdot,t)\|<\infty\quad\hbox{if }\ p>2.
 \end{cases}
$$
Then there exist $C>0$ and $(t_k)_k$ such that $t_k\to \tau$ and that 
for all $k$ we have 
$$
u(x,t_k)\le C(1+x_n)^{\beta+1} \quad  x\in  H.
$$
Moreover, in the case of $p>2$, the sequence $(t_k)_k$  may be chosen so that
$$
u(x,t_k)\le \frac{c_p}{2}x_n^{1-\beta}\quad \hbox{in }\ \Gamma_a \quad\hbox{with }\ a=\min\{1,(c_p/(2C))^{p-1}\}.
$$
 Therefore, by the local existence part of Theorems~\ref{localtheorymaxsol} and \ref{localtheorymaxsol1p2}, already proved in Section~\ref{localtheoryExistencesection}, there exists  $T=T(p,C)>0$  (cf. \eqref{localtheorytempscourt}) independent of $k$ such that the problem \eqref{localtheoryeqE1b} with initial data $u(t_k)$ possesses a unique solution $u_k$ on $[0,T)$ for all $k\in \N$.
  Fix $t_k\in ((\tau-T)_+,\tau)$ and set 
$$
\tilde u(t):=
\begin{cases}
u(t),\hspace{2cm} t\in [0,t_k]\\
u_k(t-t_k),\quad t\in [t_k,t_k+T].
\end{cases}
$$
 By uniqueness $u(t)=u_k(t-t_k)$ for $t\in[t_k,\tau)$.
Then $\tilde u$ is a solution of \eqref{localtheoryeqE1b} on $[0,t_k+T)$ and $t_k+T>\tau$, which contradicts the
definition of $\tau$. \qed

\section{ Blow-up behaviors: proof of Theorem~\ref{localtheoryonlyGBUpossible}}\label{localtheoryBA}
\begin{proof}[Proof of  Theorem~\ref{localtheoryonlyGBUpossible}$(i)$]
 Let $u_0\ge 0$ be as in the statement of Theorem~\ref{localtheoryonlyGBUpossible}. Let $u\in Y_T$ be the corresponding maximal solution with $T:=T_{\max}(u_0)<\infty$. 

  We define 
$$
v(x,t)=M(t+1)(A(t)+x_n)^\gamma,\quad A\in C^1([0,\infty)),\ A\ge 1 \ \hbox{ and } A'(t)>0.
$$
 Using $0<\gamma<\beta+1$, a simple calculation yields
$$
\begin{aligned}
(L v-|\nabla v|^p)(t+1)v^{-1}\ge&1-\gamma(\gamma-1)(A(t)+x_n)^{-2}(t+1)-\gamma^{ p} M^{p-1}(t+1)^p(A(t)+x_n)^{\gamma(p-1)-p}\\
\ge& 1-\gamma(\gamma-1)A(t)^{-2}(t+1)-\gamma^{ p} M^{p-1}(t+1)^pA(t)^{\gamma(p-1)-p}.
\end{aligned}
$$   
Hence, choosing 
$$
A(t)=\max\Big\{1,\sqrt{2\gamma(\gamma-1)_+},(2\gamma M^{\frac{p-1}{p}})^{\frac{p}{p-\gamma(p-1)}}\Big\}(t+1)^{\frac{p}{p-\gamma(p-1)}}=:C(M,p)(t+1)^{\frac{p}{p-\gamma(p-1)}}
$$ 
and using $p/(p-\gamma(p-1))\ge 1$, we obtain
$H v\ge|\nabla v|^p$. 
Since $u\in Y_T$ and $u_0\le v(\cdot,0)$, the comparison principle  Proposition~\ref{localtheoryuniqueness} yields 
$$
u(x,t)\le v(x,t)\le M(t+1)A(t)^\gamma(1+x_n)^\gamma=C(M,p)(t+1)^{1+\frac{p\gamma}{p-\gamma(p-1)}}(1+x_n)^\gamma\quad\hbox{in }\quad  H\times(0,T).
$$
This concludes the proof of \eqref{localtheoryC1-1double}.

\end{proof}

\begin{proof}[Proof of \eqref{localtheorygradlowerbound} in Theorem~\ref{localtheoryonlyGBUpossible}$(ii)$]
{\bf Step 1.} {\it Upper estimates of $|u_t|$ and $|\nabla u_t|$.}
Using Bernstein estimate \eqref{localtheorygrabound} and estimate \eqref{localtheoryC1-1double}, we get 
\be{localtheoryutil1-0}
|\nabla u(x,t)|\le C (1+x_n)^{\gamma-1}=:\phi(x_n)\quad \hbox{in }\ \big( H\setminus \Gamma_{1/2}\big)\times [T/8,T).
\ee
We claim that 
 \be{localtheorycutancienb}
|D^2 u(x,t)|+|u_t(x,t)| \le C (1+x_n)^{ p\gamma} ,\ \ \hbox{ in }\   \big( H\setminus \Gamma_{1}\big)\times[T/2,T).
 \ee
 Let $(\xi,t_0)\in\R^{n-1}\times[T/2,T)$ and $y\ge 1$. Fix $\eps\in \Big(0,\frac{ \min\{1,\sqrt{T}\}}{2}\Big)$. We denote 
 $$
 x_0=(\xi,y)\quad 
\hbox{and }\quad D:=B_\eps(x_0) \times[t_0-\eps^2,t_0].
$$
From \eqref{localtheoryC1-1double} and \eqref{localtheoryutil1-0}, we have 
$$0\le u\le C(1+y)^\gamma=:f_1(y)\quad \hbox{and } \quad |\nabla u|\le f_2(y):=\underset{ y-1/2\le s\le y+1/2}{\sup}\phi(s)\quad \hbox{in } D.$$
Setting $D':=B_{\eps/2}(x_0)\times [t_0-\eps^2/2,t_0]$, for each $q\in(1,\infty)$, by interior $L^q$ parabolic regularity, we have
 $$\begin{aligned}
 \|u\|_{W^{2,1;q}(D')}
 &\le C(n,q)\Big(\eps^{-2}\|u\|_{L^q(D)}+\||\nabla u|^p\|_{L^q(D)}\Big)\le  C(n,q,\eps)\Big(\|u\|_{L^\infty(D)}+\||\nabla u|^p\|_{L^\infty(D)}\Big)\\
 &\le \hat f(y):=C(n,q,\eps)(1+y)^\gamma,\hspace{1cm} \hbox{due to }\ p(\gamma-1)\le\gamma.
  \end{aligned}$$
Then taking $q=q(n)$ large and using standard imbeddings, we get
$$\||\nabla u|^p\|_{C^{1/2,1/4}(D')}\le C(n,p)\|u\|^p_{C^{3/2,1/4}(D')}\le C(n,p,\eps)\|u\|^{ p}_{W^{2,1;q}(D')}\le C(n,p,\eps) \hat f ^{ p}(y)$$
hence, by interior Schauder parabolic regularity, 
$$|u_t(\xi,y,t_0)|+|D^2 u(\xi,y,t_0)|\le C(n,\eps)(\|u\|_{L^\infty(D')}+\||\nabla u|^p\|_{C^{1/2,1/4}(D')})\le C(n,p,\eps) (1+y)^{ p\gamma},
$$
  which yields \eqref{localtheorycutancienb}.

We next claim that 
\be{localtheoryutil1-1}
|D u_t|\le C(1+x_n)^{ p\gamma+1},\quad \hbox{in }\ \big( H\setminus \Gamma_{1}\big)\times[T/2,T).
\ee
Let $w=\partial_{x_i} u$. By parabolic regularity, we have $w\in C^{2,1}_{loc}( H\times(0,T))$ and 
\be{localtheoryeqw}
w_t-\Delta w=b\cdot \nabla w=:F(w),\quad \hbox{with }\ b:=p|\nabla u|^{p-2}\nabla u.
\ee
Using \eqref{localtheoryutil1-0}-\eqref{localtheorycutancienb} and \eqref{localtheoryC1-1double}, we may apply interior $L^q$ and Schauder parabolic regularity and standard imbeddings similarly as above to deduce the claim.

\smallskip

\noindent{\bf Step 2.} {\it Heat semigroup regularisation in a half-space.} Let $\varphi\in L_{loc}^\infty(\bar H)$ satisfy, $|\varphi(x)|\le (1+x_n)^q$ with $q>0$. We claim that 
\be{localtheorysemigroup-prop}
|S(t)\varphi (x)|\le C(q,n)(1+x_n)^q(1+t^{q/2}),
\ee
and 
\be{localtheorynablast}
|\nabla S(t)\varphi (x)|\le C(q,n) t^{-1/2}(1+x_n)^q(1+t^{q/2})
\ee
for all $t>0$, where $(S(t))_t$ denotes the heat semigroup in $ H$ with Dirichlet boundary condition.
Using the form of $S(t)$, we have
$$
\begin{aligned}
|S(t)\varphi (x)|&\le C(n)\int_{ H}  t^{-n/2}e^{-|y-x|^2/4t}(1+y_n)^qdy\le C(n)\int_{\R^{n-1}}\int_{-x/(2\sqrt{t})}^\infty e^{-|z|^2}(1+x_n+2z_n\sqrt{t})^qdz\\
&\le C(n,q)\Big((1+x_n)^q\int_{\R^n} e^{-|z|^2}dz+(\sqrt{t})^q\int_{\R^n} |z|^q e^{-|z|^2}dz\Big),
\end{aligned}
$$
hence \eqref{localtheorysemigroup-prop}.
For \eqref{localtheorynablast}, we use the Dominated convergence theorem to obtain 
$$
\begin{aligned}
|\nabla S(t)\varphi (x)|&\le C(n)\int_ H  |x-y|t^{-(n+2)/2}e^{-|y-x|^2/4t}(1+y_n)^qdy\\
&\le C(n)t^{-1/2}\int_{\R^{n-1}}\int_{-x/(2\sqrt{t})}^\infty |z|e^{-|z|^2}(1+x_n+2|z_n|\sqrt{t})^qdz\\
&\le C(n,q)t^{-1/2}\Big((1+x_n)^q\int_{\R^n}|z| e^{-|z|^2}dz+(\sqrt{t})^q\int_{\R^n} |z|^{q+1} e^{-|z|^2}dz\Big).
\end{aligned}
$$

\noindent{\bf Step 3.} {\it Proof of \eqref{localtheorygradlowerbound}.} 
 We use a modification of the proof in \cite{PS_jmpa18} (see also \cite[Theorem~40.18*]{quittner2019superlinear}) for the bounded domain case.
From now, we denote 
$$
m(t):=\sup_{\Gamma_1}|\nabla u (\cdot,t)|, \quad T/2<t<T.
$$
$\bullet$ We first claim that $m$ is locally Lipschitz on $(T/2,T)$ and that
\be{localtheory40.40}
|m'(t)|\le \sup_{\Gamma_1}|\partial_t\nabla u(\cdot,t)|\quad \hbox{for } a.e.\quad  t\in (T/2,T).
\ee
Let $0<s<t<T$ and set $Q_{s,t}=\Gamma_1\times(s,t)$. For any $x\in \Gamma_1$, it follows from the mean value theorem that
$$
|\nabla u(x,s)|-|\nabla u(x,t)|\le |\nabla u(x,s)-\nabla u(x,t)|\le |s-t|\sup_{Q_{s,t}}|\partial_t \nabla u|
$$
hence 
$$
|\nabla u(x,s)|\le m(t)+ |s-t|\sup_{Q_{s,t}}|\partial_t \nabla u|
$$
Taking supremum for $x$ over $\Gamma_1$ and then exchanging the roles of $t$ and $s$, we get
\be{localtheory40.40a}
|m(s)-m(t)|\le |s-t|\sup_{Q_{s,t}}|\partial_t \nabla u|.
\ee
 Using the fact that $\nabla u_t\in L^\infty( \Gamma_1\times[s,t])$, it follows that the function $m$ is locally Lipschitz. 
Dividing \eqref{localtheory40.40a} by $t-s$, passing to the limit $s\to t$ with
fixed $t\in (T/2,T)$, and using the continuity of $\partial_t \nabla u$, the claim follows.

\noindent$\bullet$ We next claim that
\be{localtheory40.41}
\sup_{\Gamma_1}|\nabla u_t(\cdot,t)|\le C\Big( m(t)+1\Big)^{p-1},\quad T/2<t<T.
\ee
Let $w=u_t$, $s\in [T/4,t)$ and 
$$
M(s,t)= \sup_{\tau\in [s,t]}m(\tau),\quad \ K_1(\sigma):= \sigma^{1/2}\sup_{\Gamma_1}|\nabla w(x,s+\sigma)|,$$
\be{localtheory4.9'}
K(s,t):=\sup_{[0,t-s]}K_1(\sigma)=\sup_{\tau\in[s,t]}(\tau-s)^{1/2}\|\nabla w(\cdot,\tau)\|_{L^\infty(\Gamma_1)}.
\ee
Using the fact that $w\in C^{2,1}( \bar H\times(0,T))$ and satisfies \eqref{localtheorycutancienb}, \eqref{localtheoryeqw}, and $w=0$ on $\partial H$ for $\tau\in (0,t-s)$, we have
$$
w(x,s+\tau)=S(\tau)w(x,s)+\int_0^\tau S(\tau-\sigma)(B\cdot \nabla w)(x,s+\sigma)d\sigma,
$$
and
$$
|\nabla w(x,s+\tau)|\le  |\nabla S(\tau)w(x,s)|+ \int_0^\tau |\nabla S(\tau-\sigma)(B\cdot \nabla w)(x,s+\sigma)|d\sigma.
$$
For $x\in  H$ and $\tau\in (0,t-s)$, using \eqref{localtheoryutil1-0}-\eqref{localtheoryutil1-1} and 
\eqref{localtheorynablast}, a simple calculation using the heat kernel of $H$ similarly as in the proof of \eqref{localtheorysemigroup-prop} yields
 $$
\begin{aligned}
|\nabla S(\tau-\sigma)(B\cdot \nabla w)(x,s+\sigma)|\le&\  C(n) \int_ H |x-y|(\tau-\sigma)^{-(n+2)/2}e^{-|y-x|^2/4(\tau-\sigma)}|(B\cdot \nabla w)(y, s+\sigma)|\\
\le&\  C(n,p) m^{p-1}(s+\sigma)K_1(\sigma) \sigma^{-1/2}(\tau-\sigma)^{-1/2} \int_{\Gamma_1} \frac{|x-y|e^{\frac{-|y-x|^2}{4(\tau-\sigma)}}}{(\tau-\sigma)^{\frac{n+1}{2}}}dy\\
&+C(n,u)\int_{ H\setminus\Gamma_1} |x-y|(\tau-\sigma)^{-(n+2)/2}e^{-|y-x|^2/4(\tau-\sigma)}(1+y_n)^\delta dy\\
\le &\ C(n,p) m^{p-1}(s+\sigma)K_1(\sigma) \sigma^{-1/2}(\tau-\sigma)^{-1/2} \int_{\R^n} |z|e^{-|z|^2/4}dz\\
&+C(n,u)(\tau-\sigma)^{-1/2}(1+x_n)^\delta \Big(1+(\tau-\sigma)^{\delta/2}\Big),
\end{aligned}
$$
with some $\delta>0$.
Hence, by taking $x\in \Gamma_1$ and using \eqref{localtheorynablast}, we have 
$$
\begin{aligned}
|\nabla w(x,s+\tau)|&\le C(n,p,u)\Big(\tau^{-1/2}+ \int_0^\tau (\tau-\sigma)^{-1/2}\sigma^{-1/2}m^{p-1}(s+\sigma)K_1(\sigma)d\sigma+\int_0^\tau (\tau-\sigma)^{-1/2}d \sigma\Big)\\
&\le C(n,p,u)\Big(\tau^{-1/2}+ M^{p-1}(s,t)K(s,t)+\tau^{1/2}\Big).\\
\end{aligned}
$$
Taking the supremum over $\Gamma_1$, then
multiplying by $\tau^{1/2}$ and finally taking the supremum over $\tau\in(0,t-s)$ we get
$$
K(s,t) \le C_1+C_2(t-s)^{1/2}K(s,t) M^{p-1}(s,t).
$$
Since $C_2(t-s)^{1/2}M^{p-1}(s,t)\to 0$ as $s\to t^-$ and 
$$
\lim_{s\to T/4} C_2(t-s)^{1/2}M^{p-1}(s,t)\ge C_2(T/4)^{1/2}M^{p-1}(T/4,T/2)>0,
$$
 there exist a constant $c_0\in (0,1/2]$ independent of $t$,  and a value $s=s(t)\in (T/4,t)$ such that
\be{localtheory40.41a}
C_2(t-s)^{1/2}M^{p-1}(s,t)=c_0
\ee
and then we get 
\be{localtheory40.41b}
K(s,t)\le 2C_1.
\ee
 Together with \eqref{localtheory40.40} and \eqref{localtheory4.9'}, we obtain 
$$
|m'(\tau)|\le \sup_{\Gamma_1}|\nabla w(x,\tau)|\le 2C_1(\tau-s)^{-1/2}\quad \hbox{for }\  a.e.\quad \tau\in(s,t)
$$
By integration, for $\tau\in [s,t)$, we get
$$
 m(\tau)=m(t)-\int_{\tau}^tm'(\sigma)d\sigma\le m(t)+2C_1\int_{s}^t(\sigma-s)^{-1/2}d\sigma=m(t)+4C_1(t-s)^{1/2}.
$$
hence
\be{localtheory40.41c}
M(s,t)\le m(t)+C.
\ee
Now going back to \eqref{localtheory40.41b} and using \eqref{localtheory4.9'}, \eqref{localtheory40.41a}, we obtain
$$
\sup_{\Gamma_1}|\partial_t \nabla u (\cdot,t)|\le 2C_1 (t-s)^{-1/2}\le 2c_0^{-1}C_1C_2 M^{p-1}(s,t),
$$
 and \eqref{localtheory40.41} follows from \eqref{localtheory40.41c}.

\noindent $\bullet $ \eqref{localtheory40.41} along with \eqref{localtheory40.40}  yields
$$
m'(t)\le C(m(t)+1)^{p-1},\quad\hbox{for } a.e.\ t\in (T/2,T).
$$
Integrating over $(t,\tau)$ and using $m(\tau)\to \infty$ as $\tau\to T$ we arrive at
$$
m(t)+1\ge C (T-t)^{-\frac{1}{p-2}},
$$
which concludes the proof.
\end{proof}

\begin{proof}[Proof of \eqref{localtheoryuppergradbound} in Theorem~\ref{localtheoryonlyGBUpossible}$(ii)$] 
{\bf Step 1. Control of $u_t$.} We shall prove that 
\be{localtheoryclaim1.0}
|u_t(x,t)|\le C(n,p,u)(1+x_n)^{ p\gamma}\quad \hbox{in }\quad H\times (T/2,T).
\ee 
 Recalling  \eqref{localtheorycutancienb}, it remains to prove that \eqref{localtheoryclaim1.0} holds in $\Gamma_1\times[T/2,T)$. Let first note that for any $0<t_1<t_2<T$, we have 
$$
u,\nabla u\in L^\infty\Big(\Gamma_1\times[t_1,t_2]\Big).
$$
Then by the parabolic regularities as in the proof of \eqref{localtheorycutancienb}, we have 
$$
D^2u, u_t\in  L^\infty\Big(\Gamma_1\times[t_1,t_2]\Big).
$$
Set $V=u_t$, we have 
$$
\begin{cases}
V_t-\Delta V=b\cdot \nabla V\quad\hbox{in }\Gamma_1\times(T/2,T)\\
V(\cdot,T/2) \in  L^\infty(\Gamma_1)\\
V(x',0,t)=0,\quad |V(x',1,t)|\le C_0.
\end{cases}
$$
Since $V,b\in L^\infty\Big(\Gamma_1\times[T/2,T-\eps]\Big)$ for any $\eps>0$, the maximum principle \cite[Proposition~52.8]{quittner2019superlinear} (cf. Proposition~\ref{localtheory52.8} below) yields 
$$
\|V(\cdot,t)\|_{L^\infty(\Gamma_1)}\le \max\Big\{C_0,\|V(\cdot,T/2)\|_{L^\infty(\Gamma_1)}\Big\}\quad\hbox{for all }\ t\in (T/2,T-\eps).
$$
The conclusion then follows by letting $\eps\to 0$.

\smallskip

\noindent {\bf Step 2. Partial control of $\nabla u$ on $\Gamma_1$.} We claim that there exists $C>0$ such that 
\be{localtheoryc1-31}
|\nabla u(x,t)|\le C x_n^{-\beta} , \quad (x,t)\in \Gamma_1\times [T/2,T).
\ee
Let $x_0\in \Gamma_1$ and any $\tau \in (T/2,T)$. We define here
$$
R=x_{0,n}/2,\quad B_R=\{x\in \R^n,\ |x-x_0|<R\}.
$$
The proof is based on a Bernstein type argument, similar to \cite[Section~3]{chso2} adapted to the present context, which is given in Appendix~\ref{localtheorylemma3.1}. 
Namely, let $\delta\in(0,1)$, $\tau\in (0,T)$, $w=|\nabla u|^2$ and $z=\eta w$, where $\eta$ is the cut-off from \eqref{localtheorydefeta1}-\eqref{localtheoryeta}. By \eqref{localtheoryeqL} with $f(v)=v$, we have 
$$
\mathcal{L}z+\eta|D^2v|^2\leq
C\eta^\delta(\tau^{-1}+R^{-2}) w+ CR^{-1}\eta^\delta w^{(p+1)/2}\quad\hbox{in }\ \tilde{Q}:=B_R\times(\tau/4,\tau),
$$
where 
$$
\mathcal{L}:=\partial_t -\Delta-b\cdot \nabla\quad \hbox{and }\quad b=p|\nabla u|^{p-2}\nabla u.
$$
Note that 
$$
|w^{p/2}-u_t|=|\Delta u|\le \sqrt{n}|D^2 u|
$$
which implies
$$
\frac{w^p}{2n}\le |D^2 u|^2+|u_t|^2.
$$
 Choosing $\delta=(p+1)/(2p)$ and using \eqref{localtheoryclaim1.0} together with Young's inequality, we obtain
$$
\mathcal{L}z+\frac{\eta w^p}{2n}\leq \frac{\eta w^p}{4n}+
C(\tau^{-1}+R^{-2})^{p/(p-1)}+CR^{-2p/(p-1)}+C=:\frac{\eta w^p}{4n}+B.
$$
Therefore, we have 
$$
\mathcal{L}z\le 0\quad\hbox{in }\ \{(x,t)\in \tilde Q: z\ge (4nB)^{1/p}\}.
$$
Since $z=0$ on the parabolic boundary of $\tilde Q$,
 the maximum principle yields $z\le (4nB)^{1/p}$ in $\tilde Q$.
Finally, using $z=\eta|\nabla u|^2$ and the definition of $\eta$, this implies
$$
|\nabla u (x,t)|\le C(n,p,u)\Big(R^{-\beta}+1 \Big)\le  C(n,p,u)R^{-\beta}\quad\hbox{in }\quad B_{R/2}\times[\tau/2,\tau]
$$
hence \eqref{localtheoryc1-31}.

\vspace{2pt}

\noindent {\bf Step 3. Conclusion.} \label{localtheorystep32.7}
 We use a rescaling argument similar to \cite[Theorem~1.3]{FPS2020}.
Assume that \eqref{localtheoryuppergradbound} fails, then there exist $\eps>0$ and a sequence $(x_j,t_j)\in \Gamma_1\times[T/2,T)$ such that 
\be{localtheorycontrcontrolbor}
|\nabla u (x_j,t_j)|\ge (1+\eps)d_p \rho_j^{-\beta}+j
\ee
where $\rho_j=x_{j,n}$, with $x_j=(x_j',x_{j,n})\in \Gamma_1$.
We note that \eqref{localtheoryc1-31} implies $\rho_j\to 0$ as $j\to \infty$. 
Setting $z_j=(x_j',0)$ and
$$
v_j(y,s)=\rho_j^{\beta-1}u(z_j+\rho_j y,t_j+\rho_j^2s)\quad\hbox{for } (y,s)\in D_j=\big(\rho_j^{-1}\Gamma_1\big)\times (-t_j\rho_j^{-2},0],
$$
a direct computation leads to 
\be{localtheoryeqbeforlim}
\begin{cases}
\partial_s v_j-\Delta_{y} v_j=|\nabla_y v_j|^p\quad \hbox{in } D_j,\\
v_j(y',0,s)=0,\  -t_j\rho_j^{-2}< s\le 0.
\end{cases}
\ee
Moreover, by using \eqref{localtheoryc1-31}, for each $A>0$ there exists $j_0=j_0(A)\ge 1$ such that 
$$
|\nabla v_j(y,s)|\le C\rho_j^\beta(\rho_j y_n)^{-\beta}\le C y_n^{-\beta}, \quad y\in \Gamma_A,\ s\in (-t_j\rho_j^{-2},0],\ j\ge j_0.
$$
Since $\beta<1$ owing to $p>2$, by integrating and using the boundary condition, this implies 
$$
| v_j(y,s)|\le C(1-\beta)^{-1} y_n^{1-\beta},\quad y\in \Gamma_A,\ s\in (-t_j\rho_j^{-2},0],\ j\ge j_0.
$$
 By this together with \eqref{localtheoryeqbeforlim} and interior parabolic estimates,
 there exists a subsequence of $(v_j)_j$ which converges in {  $C^{2,1}_{loc}( H \times(-\infty,0])$}, 
to an ancient (nonnegative) solution { $V\in C^{2,1}(\R^n_+\times(-\infty,0])\cap C(\bar  H \times(-\infty,0])$} of 
\be{localtheoryeqlim1}
\begin{cases}
V_s-\Delta V=|\nabla_y V|^p\quad \hbox{in } \R^n_+\times(-\infty,0],\\
V(y,s)=0,\quad y\in\partial \R^n_+, s\in(-\infty,0].
\end{cases}
\ee

On the other hand, \eqref{localtheoryclaim1.0} implies that 
$$
|\partial_s v_j(y,s)|\le C\rho_j^{\beta+1}\quad\hbox{in }\quad D_j.
$$
In particular, passing to the limit yields $\partial_sV=0$.
Therefore, by the Liouville type result \cite[Theorem~1.1]{FPS2020}, it turns out that there exists $a>0$ such that
$$
V\equiv 0\quad \hbox{or }\quad V(y)=c_p\bigl((y_n+a)^{1-\beta}-(a)^{1-\beta}\bigr).
$$
Hence, setting $e_n=(0,\cdots,0,1)$, we have $ |\nabla V(e_n)|=d_p (1+a)^{-\beta}\le d_p$.
  But by using  \eqref{localtheorycontrcontrolbor} for large $j$, we have
$$
|\nabla v_j( e_n,0)|=\rho_j^{\beta}|\nabla u( x_j,t_j)|\ge (1+\eps)d_p,
$$
therefore, by letting $j\to \infty$, we get $(1+\eps)d_p\le  |\nabla V( e_n)|\le d_p$, a contradiction.
 This concludes the proof.
\end{proof}

\section{Proof of Theorems~\ref{localtheoryapplication}-\ref{localtheorydichotomie}}\label{localtheoryAPL}

A main ingredient of the proof of these theorems is the 
following a priori estimate which holds for any (classical) solution  of \eqref{localtheoryeqE0} in a half-space with Dirichlet boundary condition. 

\begin{prop}\label{localtheorythmancient3}
Let $p>2$ and $u\ge0$ be any solution of \eqref{localtheoryeqE1b} in $ Q_T:=\R^n_+\times(0,T).$
 Then we have
 \be{localtheorycontroludsqa}
0\le  u(x,t)\le C_0(n,p)
x_n^{1-\beta} \left(1+x_n^{2\beta}(T-t)^{-\beta}+\min\{x_n^2t^{-1},(T-t)/t\}\right),
 \ee
  and 
  \be{localtheorycontroludsq1a}
 |\nabla u(x,t)|\le C_1(n,p) x_n^{-\beta} \left(1+x_n^{2\beta}(T-t)^{-\beta}+x_n^{2\beta}t^{-\beta}+\min\{x_n^2t^{-1},(T-t)/t\}\right).
 \ee
 In particular, we have 
 \be{localtheoryusful1}
  |\nabla u(x,t)|\le 2C_1(n,p) x_n^{-\beta} \left(1+x_n^{2\beta}(T-t)^{-\beta}\right)\quad \hbox{in }\ \R^n_+\times[T/2,T).
 \ee
\end{prop}
 The proof, which is a modification of \cite[Theorem~6]{chso2}, is given in Appendix~\ref{localtheoryAPE}.

\begin{proof}[Proof of Theorem~\ref{localtheorydichotomie}] Based on Proposition~\ref{localtheorythmancient3}, the proof of \eqref{localtheorydico1} follows the argument in Step 3 of the proof of \eqref{localtheoryuppergradbound}; see page~\pageref{localtheorystep32.7}. For the reader's convenience, we provide a complete proof.

Assume by contradiction that the conclusion of Theorem~\ref{localtheorydichotomie} fails. Which means that there exist $\eps>0$ and a sequence $(x_j,t_j)\in  H\times[T/2,T)$ such that 
\be{localtheorycontrI-II}
\begin{cases}
|\nabla u (x_j,t_j)|\ge (1+ \eps)d_p \rho_j^{-\beta}+j\\
|\nabla u (x_j,t_j)|\ge C_1 j \rho_j^{\beta}(T-t_j)^{-\beta},
\end{cases}
\ee
where $\rho_j=x_{j,n}$, with $x_j=(x_j',x_{j,n})\in  H$, and $C_1$ is the positive constant defined in Proposition~\ref{localtheorythmancient3}.
 We claim that 
\be{localtheory0clai0m1}
\rho_j\to 0,\quad (T-t_j)\rho_j^{-2}\to \infty\quad \hbox{as }\ j\to \infty.
\ee
We first show that $\rho_j\le 1 $ for sufficiently large $j$. Otherwise, using \eqref{localtheoryusful1} together with the second line of \eqref{localtheorycontrI-II}, we obtain
$$
 (j-2)\rho_j^\beta(T-t_j)^{-\beta}\le 2,
$$
which is impossible for large $j$. Hence $\rho_j\le 1$ for sufficiently large $j$.
 Moreover,  \eqref{localtheoryusful1} 
and the second line of \eqref{localtheorycontrI-II} yields 
$$
j \rho_j^{\beta}(T-t_j)^{-\beta}\le 2\rho_j^{-\beta}+2\rho_j^{\beta}(T-t_j)^{-\beta}.
$$
Therefore,
$$
\rho_j^2\le \frac{2^{p-1}(T-t_j)}{(j-2)^{p-1}} \quad \hbox{for }\ j>2.
$$
This in particular implies \eqref{localtheory0clai0m1}

  Let $z_j=(x_j',0)$ and define
$$
v_j(y,s)=\rho_j^{\beta-1}u(z_j+\rho_j y,t_j+\rho_j^2s)\quad\hbox{for } (y,s)\in D_j=\R^n_+\times (-t_j\rho_j^{-2},(T-t_j)\rho_j^{-2}).
$$
By direct computation, we have 
\be{localtheoryeqseq}
\begin{cases}
\partial_s v_j-\Delta_{y} v_j=|\nabla_y v_j|^p\quad \hbox{in } D_j,\\
v_j(y,s)=0,\quad y\in\partial \R^n_+,\  -t_j\rho_j^{-2}< s<(T-t_j)\rho_j^{-2}.
\end{cases}
\ee
On the other hand, by using $z_{j,n}=0$ and \eqref{localtheoryusful1}, we have 
$$
|\nabla v_j(y,s)|\le 2C_1\rho_j^\beta\left( (\rho_j y_n)^{-\beta}+ \Big(\frac{\rho_j y_n}{T-t_j-\rho_j^2s}\Big)^{\beta}\right)\le2 C_1\left( y_n^{-\beta}+ y_n^{\beta}\Big(\rho_j^{-2}(T-t_j)-s\Big)^{-\beta}\right)\ \hbox{ in }\ D_j
$$
which, by using  $p>2$ and the boundary conditions, implies 
$$
| v_j(y,s)|\le 2C_1(1-\beta)^{-1} \left( y_n^{1-\beta}+ y_n^{1+\beta}\Big(\rho_j^{-2}(T-t_j)-s\Big)^{-\beta}\right).
$$
 By \eqref{localtheory0clai0m1}, for all $A>0$, there exists $j_0:=j_0(A)> 2$ such that $\R^n_+\times [-A,A]\subset D_j$ for all $j\ge j_0$ and,
$$
|\nabla v_j(y,s)|\le 2C_1\left(y_n^{-\beta}+y_n^\beta\right),\ |v_j(y,s)|\le 2C_1(1-\beta)^{-1}\left(y_n^{1-\beta}+y_n^{1+\beta}\right),\quad (y,s)\in \Gamma_A\times [-A,A],\ j\ge j_0.
$$
Hence $|\nabla v_j|, v_j\in L^\infty(\Gamma_{\eta,R}\times[-A,A])$ for any $0<\eta<R<\infty$.
By interior parabolic estimates, we conclude that $(v_j)_{j\ge j_0}$ is precompact in $C^{2,1}(\Gamma_{\eta,R}\times[-A,A])$.
We then deduce that some subsequence of $(v_j)_j$ converges in each $C^{2,1}(\Gamma_{\eta,R}\times[-A,A])$ to an entire nonnegative solution $V\in C^{2,1}(\R^n_+\times(-\infty,\infty))\cap C(\bar\R^n_+\times(-\infty,\infty))$ of 
\be{localtheoryeqseq2}
\begin{cases}
V_s-\Delta V=|\nabla_y V|^p\quad \hbox{in } \R^n_+\times(-\infty,\infty),\\
V(y,s)=0,\quad y\in\partial \R^n_+, s\in\R.
\end{cases}
\ee
By the Liouville type result \cite[Theorem 3]{chso2}, it turns out that there exists $a\ge 0$ such that
$$
V(y,s)=V(y)=c_p\bigl((y_n+a)^{1-\beta}-a^{1-\beta}\bigr).
$$
 Hence, setting $e_n=(0,\cdots,0,1)$, we have $|\nabla V(e_n)|=d_p (1+a)^{-\beta}\le d_p$.
 Using the first line of \eqref{localtheorycontrI-II}, we also have
$$
|\nabla v_j(e_n,0)|=\rho_j^{\beta}|\nabla u(x_j,t_j)|\ge (1+\eps)d_p,
$$
therefore, we must have $(1+\eps)d_p\le |\nabla V(e_n,0)|= V'(1)\le d_p$, a contradiction.
 This concludes the proof.
\end{proof}

\begin{proof}[Proof of Theorem~\ref{localtheoryapplication}]
Let $M(x,t):=|\nabla u(x,t)|^{p-1}$. Our assumption implies there exist $(x_j,t_j)\in\Gamma_1\times[T/2,T)$, with $(x_{j,n},t_j)\to(0,T)$, such that
\be{localtheoryaprouv}
M(x_j,t_j)> \frac{2 j}{T-t_j}.
\ee
Let $D=\bar\Gamma_1\times[T/2,T)$ and $\Sigma=\bar{\Gamma}_1\times[T/2,T]$. Then $\Gamma:=\Sigma\backslash D=\bar \Gamma_1\times \{T\}$ and for all $(x,t)\in D$ we define $d_p(x,t):=d((x,t),\Gamma)=T-t$.\\
By the doubling lemma in \cite{PQS2}, 
 there exists $(y_j,s_j)\in D$ such that 
\be{localtheorylem5.1}
\begin{cases}
M(y_j,s_j)\ge M(x_j,t_j)\\
M(y_j,s_j)>\frac{2j}{ d_p(y_j,s_j)}\\
M(x,t)\le 2 M(y_j,s_j),\quad \hbox{for all } (x,t)\in D\times B_j,
\end{cases}
\ee
where 
$$
B_j:=\Big\{ (x,t)\in \R^{n+1},\ |x-y_j|+|t-s_j|^{1/2}\le j\lambda_j\Big\},\quad \hbox{ with  }\lambda_j=M^{-1}(y_j,s_j)\underset{j\to\infty}{\to}0.
$$
 In particular, $s_j\to T$ as $j\to \infty$.
 For clarity, we rewrite the second line of \eqref{localtheorylem5.1} as
\be{localtheoryclaimimport}
|\nabla u(y_j,s_j)|^{p-1}> \frac{2 j}{T-s_j}.
\ee
We define the rescaled solutions
$$
v_j(y,s)=\lambda_j^{\beta-1}u(\lambda_jy+y_j,s_j+\lambda_j^2s)\qquad (y,s)\in D_j\cap \tilde{B_j},
$$
where
$$
\begin{aligned}
 D_j:= \tilde D_j\times\big(-s_j\lambda_j^{-2},&(T-s_j)\lambda_j^{-2}\big),\  \tilde D_j:=\{y\in\R^n :-\rho_j\lambda_j^{-1}<y_n\le \lambda_j^{-1}(1-\rho_j) \},\\
  \rho_j:=d(y_j,\R^n_+)=y_{j,n}\quad & \vspace{2cm}\hbox{and }\tilde B_j:=\Big\{(y,s)\in \R^{n+1},\ \ |y|+|s|^{1/2}\le j\Big\}.
\end{aligned}
$$
By a direct computation, we have 
\be{localtheoryequinj}
\partial_s v_j-\Delta_{y} v_j=|\nabla_y v_j|^p,\ (y,s)\in D_j\cap \tilde{B_j},\quad\hbox{with } v_j(y',-\rho_j\lambda_j^{-1},s)=0.
\ee
and, using the last line of \eqref{localtheorylem5.1}, it follows that 
\be{localtheorycboundcond}
 |\nabla v_j(y,s)|\leq 2^{1/(p-1)}\  \hbox{ and $v_j(y,s)\le2^{1/(p-1)}\Big(y_n+\rho_j\lambda_j^{-1}\Big)$ in }\ D_j\cap \tilde{B_j}.
\ee
 Combining \eqref{localtheoryclaimimport} with \eqref{localtheoryusful1} in Proposition~\ref{localtheorythmancient3}, it follows that
$$
 2j(T-s_j)^{-1}\le 
\lambda_j^{-1}\le  C(n,p)\left( \rho_j^{-1}+(T-s_j)^{-1}\right).
$$
Therefore, for sufficiently large $j$, we have $(T-s_j)^{-1}\le \rho_j^{-1}$, hence in particular $\rho_j\to0$ and 
$$
\rho_j \lambda_j^{-1}\le C(n,p),
$$
and we also have 
$$
(T-s_j)\lambda_j^{-2}\to \infty,\quad\hbox{ as } j\to \infty\ .
$$
Consequently, after extracting a subsequence, there exists $a\ge 0$ such that
$$
\tilde D_j\to \tilde{D}_\infty=\left(\R^n_+-A\right) \quad\hbox{as $j\to \infty$ with } A=(0,\cdots,0,a).
$$
Together with \eqref{localtheoryequinj}-\eqref{localtheorycboundcond} and \eqref{localtheorycontroludsqa}, interior-boundary parabolic estimates and standard imbeddings, we obtain a subsequence of $v_j$ which converges in $C^{2,1}_{loc}(\bar  H \times\R)$, to a nonnegative entire solution~$U_1\in C^{2,1}(\tilde{D}_\infty\times\R)\cap C(\bar{\tilde{ D}}_\infty\times\R)$ of  \eqref{localtheoryeqseq2} satisfying $|\nabla U_1 (0,0)|=1$ {\cb with  $U_1(x',-a,t)=0$}. 
By \cite[Theorem 3]{chso2}, we know that 
$$
U_1(x,t)=U_1(x_n)=c_p\bigl((x_n+\beta)^{1-\beta}-(\beta-a)^{1-\beta}\bigr).
$$
In particular, $\nabla v_j$ converges to $ U_1' e_n$ locally uniformly as $j\to \infty$.
\end{proof}

\appendix

 \section{Auxiliary results}
 
We first collect some known auxiliary results used throughout the paper and refer to \cite{chso2} for the proofs.
For given $R,T>0$, we denote
$$ B_R:=\{x\in\R^n;\  |x|<R\},\quad 
Q_{R,T}:=B_R \times (0, T),\quad Q'_{R,T}:=B_{R/2} \times (0, T).$$
The first one is the following local Li-Yau type estimate.
 \begin{thm}[Theorem~7 of \cite{chso2}]  \label{localtheorythmLY}
 Let  $p>2$,  $R, T>0$ and let $u\in C^{2,1}(Q_{R,T})$ be a classical solution of \eqref{localtheoryeqE0}
 in~$Q_{R,T}$. 
 
\begin{itemize}
\item[(i)]  For each $a \in[0, 1)$, there exists $C = C(n,p,a) > 0$ such that
\be{localtheorycontrolul0}
a|\nabla u|^p-u_t\le  C\bigl(R^{-\beta-1}+R^{1-\beta}t^{-1}\bigr)
\quad\hbox{ in $Q'_{R,T}$.}
\ee

\item[(ii)] There exists $C = C(n,p) > 0$ such that, for all $0<t<s<T$ and $x,y\in B_{R/2}$,
\be{localtheorycontrolul}
 u(x,t)\le u(y,s)+C\Bigl(\frac{|y-x|^p}{s-t}\Bigr)^\beta+C\bigl(R^{-\beta-1}+R^{1-\beta}t^{-1}\bigr)(s-t),
 \ee
\end{itemize}
 \end{thm} 
 The second one is the following local Bernstein type estimate.
\begin{thm}[Theorem~8 of \cite{chso2}] \label{localtheorypropBern}
 Let $p > 1$, $R, T>0$ 
 and let $u\in C^{2,1}(Q_{R,T})$  be a classical solution of \eqref{localtheoryeqE0} 
 in $Q_{R,T}$ such that
 $M:=\sup_{Q_{R,T}} u<\infty$. We have:
\be{localtheorygrabound} 
 |\nabla u|\le C(n,p)\left\{\frac{M-u}{R}+\Big(\frac{M-u}{R^2\wedge t}\Big)^{1/p}\right\}
\ \hbox{ in $Q'_{R,T}$.}
\ee 
 \end{thm}

We close this section with the following two maximum/comparison principles used in our proofs. We refer the reader to \cite[Section~52.3]{quittner2019superlinear} for further details. In what follows, we denote $Q_T=\Omega\times(0,T)$.

\begin{prop}[Proposition~52.8 of \cite{quittner2019superlinear}]\label{localtheory52.8}
Let $0<T<\infty$, $K\ge0$, and  $\Omega$ be an arbitrary domain in $\R^n$.
Assume that $w\in C(\overline{\Omega}\times(0,T))\cap C([0,T),L^2_{loc}(\overline{\Omega}))$, satisfies 
 $$
 \sup_{Q_T}w<\infty,\quad  w_t,\nabla w,D^2w\in L^2_{loc}(Q_T).
 $$
If $w\le A$ on the parabolic boundary of $Q_T$ and 
$$
w_t-\Delta w\le K|\nabla w|\quad a.e. \ \hbox{in } Q_T,
$$
 then $w\le A$ in $Q_T$.
\end{prop}

\begin{prop}[Proposition~52.10 of \cite{quittner2019superlinear}]\label{localtheory52.10}
Let $0<T<\infty$, $\Omega$ be an arbitrary domain in $\R^n$, and let $f=f(s,\xi):\R\times\R^n\to \R$, be a $C^1$-function. Let $u\in C(\overline{\Omega}\times(0,T))$ satisfy
$$
u\in C([0,T),L^2_{loc}(\overline{\Omega})),\quad u\in L^\infty(Q_T),\quad u_t,\nabla u,D^2u\in L^2_{loc}(Q_T),
$$
and similarly for $v$. If $f$ depends on $\xi$, we also assume that $\nabla u,\nabla v\in L^\infty(Q_T)$. If $u\le v$ on the parabolic boundary of $Q_T$ and 
$$
u_t-\Delta u-f(u,\nabla u)\le v_t-\Delta v-f(v,\nabla v)\quad a.e. \ \hbox{in } \ Q_T, 
$$
then $ u\le v$ in $Q_T$.
\end{prop}

\section{Proof of Lemma~\ref{localtheoryBernchso}}\label{localtheorylemma3.1}

The proof follows the lines of the proof of \cite[Theorem~8]{chso2}. We will give the main ideas of the proof.
Here and in the rest of the proof, for given $R>0$, we denote
$$ 
B_R:=\{x\in\R^n;\  |x|<R\},\quad 
Q_{R,\tilde T}:=(B_R\cap \Omega) \times (0,\tilde T),\quad Q'_{R, \tilde T}:=(B_{R/2} \cap\Omega)\times (0, \tilde T).
$$
Let $f$ be a function, to be determined later, such that
\be{localtheoryfmap}
\begin{aligned}
&\hbox{$f\in C^1([0,\infty))\cap C^3(0,\infty)$,\quad $f'>0$ on $(0,\infty)$,}\\
&\hbox{$f$ maps $[0,\infty)$ onto $[-M,\infty)$.}
\end{aligned}
\ee
Let us put
$$v=f^{-1}(-u),\qquad w=|\nabla v|^2.$$

Let $\delta\in (0,1)$, $\tau\in(0,\tilde T)$ and let $\eta=\eta(x,t):=\eta_1(t)\tilde{\eta}(x)$ 
where  $\eta_1\in C^1([\tau/4,\tau])$ is a cut-off in time such that:
\be{localtheorydefeta1}
0\le \eta_1\le 1,\ \ \eta_1(\tau/4)=0, \ \ \eta_1\equiv 1 \ \hbox{in } [\tau/2,\tau),
\ \ |\eta_1'(t)|\le C\tau^{-1}\eta_1^\delta(t).
\ee
Let $\rho=2R/3$. We can select a (space) cut-off function $\tilde\eta\in C^2(\bar{B}_R)$,
with 
$$
0<\tilde\eta\le 1\ \hbox{for $|x|<\rho$},\quad \tilde\eta=1\ \hbox{for $|x|\le R/2$}, 
\quad \tilde\eta=0\ \hbox{for $|x|\ge\rho$,}
$$
 and
\be{localtheoryeta}
|\nabla \tilde\eta|\le C R^{-1}\tilde\eta^{\delta},\quad |\Delta \tilde\eta|+\tilde \eta^{-1}|\nabla \tilde\eta|^2\le C R^{-2}\tilde \eta^{\delta}
\ee
with $C=C(\delta)>0$.
 Put
$$z=\eta w$$
and set $\tilde Q:=(B_{3R/4}\cap\Omega)\times(\tau/4,\tau]$.
In the rest of the proof,
all computations will take place in $\tilde Q\cap\{z>0\}$, unless otherwise specified, and 
$C$ will denote a generic positive constant depending only on $p,n,\gamma$.
From \cite[Step 1 and 2 of Proof of Theorem~8]{chso2}, we have 
\be{localtheoryeqL}
\mathcal{L}z+\eta|D^2v|^2\leq \eta\, \mathcal{N}w +
C\eta^\gamma(\tau^{-1}+R^{-2}) w+ CR^{-1}\eta^\gamma\Bigl|p{f'}^{p-1}\,w^{(p+1)/2}-2\frac{f''}{
f'}w^{3/2}\Bigr|,
\ee
where 
$$\mathcal{L}z:= z_t-\Delta z+b\cdot\nabla z,
\quad b:=\Bigl[p{f'}^{p-1}\,w^\frac{p-2}{ 2}-2\frac{f''}{ f'}\Bigr]\nabla v,$$
\be{localtheorydefNw}
\mathcal{N}w:=
-2({f'}^{p-1})'\,w^{(p+2)/2}+2\displaystyle\Bigl(\frac{f''}{ f'}\Bigl)'\,w^2.
\ee

Following the same argument as in \cite[Step 3 of the proof of Theorem~8]{chso2}, it turns out that, with the (optimal) choice of $f$ given by:
\be{localtheorychoicef}
f(v)=h^{-1}(v)-M\quad \hbox{i.e.}\quad v=h(M-u),
\ee
where
\be{localtheorychoiceh}
h(s)=\int_0^s\frac{d\tau}{g(\tau)},\quad g(s)=R^{-1}s+Ks^\theta,\quad \hbox{with $\theta=1/p$ and $K=(R^{-2\theta}+\tau^{-\theta})$},
\ee 
 there exists $A=A(n,p)>~0$ such that 
$$
\mathcal{L} z\leq 0
\quad\hbox{ in $\{(x,t)\in \tilde Q;\ z(x,t)\geq A\}$}.
$$
 Note that $v\geq 0$ in $Q_{R,T}$ due to $u\leq M$ and $g>0$ in $[0,\infty)$.
On the parabolic boundary of $\tilde Q$, using our assumption \eqref{localtheoryhypoth} and the choice of auxiliary function \eqref{localtheorychoiceh}, we have 
$$
\begin{aligned}
z=&\eta|\nabla v|^2= \Big(\frac{|\nabla u|}{g(M-u)}\Big)^2\eta\le \Big(\frac{\tilde{M}(\tau/4)^{-\gamma_0}}{g(M)}\Big)^2=:M_2^2
\end{aligned}
$$
Therefore, the maximum principle yields $z\le \max\{A,M_2^2\}$ in $\tilde Q$.
Finally, using $z=\eta|\nabla v|^2$, \eqref{localtheorychoiceh}, \eqref{localtheorychoicef} and the definition of $\eta$, this implies
$$
\begin{aligned}
|\nabla u|&=|\nabla v|\, g(M-u)\le (\max\{A,M_2^2\})^{1/2} g(M-u)
\le C(n,p)\big(g(M)+\tilde{M}\tau^{-\gamma_0}\big)\frac{g(M-u)}{ g(M)}
\end{aligned}
$$ 
in $B_{R/2}\times[\tau/2,\tau]$.
For each $t\in (0,\tilde T]$, applying this 
with $\tau=t$, $R=1$ and using $u\ge 0$ and  \eqref{localtheorychoiceh} yields the conclusion.

\section{A priori estimates: proof of Proposition~\ref{localtheorythmancient3}}\label{localtheoryAPE}
 \label{localtheorySecPos}
 The proof of Proposition~\ref{localtheorythmancient3} is based on modifications of the arguments in \cite[Section~5]{chso2} written there for ancient solutions. We include it here  with the necessary changes. 
 To this end, since our proof is valid for any local in time solution, we consider the following problem
 \begin{equation}\label{localtheoryeqE1}
	\begin{cases}
	u_t-\Delta u=|\nabla u|^p,&(x,t)\in Q:=\R^n_+\times (0,T),\\
	\noalign{\vskip 1mm}
	u=0,&(x,t)\in \partial\R^n_+\times (0,T).\\
	\end{cases}
	\end{equation}
	
\subsection{Integral estimates}\label{localtheorySubSecInt}
 
We here provide preliminary integral estimates for $u$.
In this subsection, we fix $\alpha\in (1,p-1)$, and $C, C_i$ denote generic positive constants depending only on $n,p,\alpha$.

 \begin{lem}\label{localtheory1L3}
 Under assumption of Proposition~\ref{localtheorythmancient3}, for any $a\in \R^{n-1}$, we have 
\be{localtheoryupint}
\int_{B_3'(a)\times(0,3)} u(t)\vap \,dx\le C\big(1+(T-t)^{-\beta}\big),\quad 0<t<T.
\ee
 \end{lem}

  In view of its proof we may assume $a=0$ without loss of generality and we define the cylinders
$$\Omega=B'_3\times(0,3),\qquad
\Omega_\eps=\Omega\cap\{x_n>\eps\},\quad \eps>0.$$
Since $u$ is assumed only continuous (and not $C^1$) up to $x_n=0$, we will argue on $\Omega_\eps$.
For the sake of passing to the limit $\eps\to 0$, we will make use of the Bernstein estimate in \eqref{localtheorygrabound}. 
Namely, since $p>2$, the latter guarantees that
\be{localtheoryBernSZ}
|\nabla u(x,t)|\le C(M_0+1)(x_n^{-1}+t^{-1/p}),\quad x\in B'_3\times(0,3),\ 0<t_0<t<t_1<T, 
\ee 
where $M_0:=\underset{  B'_4\times(0,4)\times(t_0,t_1)}\sup |u|<\infty.$
We also fix a cylindrical cut-off function $\varphi$ such that
\be{localtheorydefvarphi}
\begin{aligned}
&\varphi(x):=\psi(x')\chi(x_n),\quad\hbox{where $\chi (s)=(3-s)s^\alpha$ and}\\
 &\psi\in C^2(B_3'),\ 0\le\psi\le 1,\ |\nabla \psi|\le \psi^{1/p},
 \ \psi\equiv 1\mbox{ in $B_2'$},\ \psi=0 \mbox{ on $\partial B'_3$}.
 \end{aligned}
 \ee
With this cut-off,  we shall use the following properties, 
 including the weighted Poincar\'e type inequality \eqref{localtheoryeuphi2}.
 
 \begin{lem}\label{localtheorygenlem}
For all $w\in W^{1,p}(\Omega_\eps)\cap C(\overline{\Omega_\eps})$, we have
\be{localtheoryeuphi2a}
\Big|\int_{\Omega_\eps}\nabla w\nabla \vap \,dx\Big|\le \frac12 \int_{\Omega_\eps}|\nabla w|^p\vap \,dx +C
\ee
and
\be{localtheoryeuphi2}
\int_{\Omega_\eps} |w|^p \psi(x') \,dx\le C\Bigg(\int_{\Omega_\eps}|\nabla w|^p\vap \,dx+\int_{B'_3} |w(x',\eps)|^p \psi(x')\,dx'\Bigg).
\ee
\end{lem}

\begin{proof} 
We refer to \cite[Lemma~5.1]{chso2} for the proof of Lemma~\ref{localtheorygenlem}
\end{proof}

\begin{proof}[Proof of Lemma~\ref{localtheory1L3}]

Let $\eps\in(0,1)$ and set $y_\eps(t)=\int_{\Omega_\eps} u(t)\vap \,dx$. Multiplying the PDE in \eqref{localtheoryeqE1} by $\varphi$, integrating by parts over $\Omega_\eps$
and using $\varphi=0$ on $\{x_n>\eps\}\cap\partial\Omega_\eps$, we get
$$\begin{aligned}
y'_\eps(t)
&= \int_{\Omega_\eps}|\nabla u|^p\vap \,dx+\int_{\Omega_\eps} \vap\Delta u \,dx\\
& =\int_{\Omega_\eps} |\nabla u|^p\vap \,dx-\int_{\Omega_\eps}\nabla u\nabla \vap \,dx+\int_{\{x_n=\eps\}}  \vap u_\nu \,dx'.
 \end{aligned}$$
 Next set 
 \be{localtheorydefIJeps}
 I_\eps(t):=\int_{B'_3} \vap |u_\nu(x',\eps,t)| \,dx',\quad J_\eps(t):=C\int_{B'_3} |u(x',\eps,t)|^p \psi(x')\,dx'.
 \ee
 Using \eqref{localtheoryeuphi2a},  \eqref{localtheoryeuphi2}, $\varphi\le C\psi$ and H\"older's inequality, we obtain
$$\begin{aligned}
y'_\eps(t)\
&\ge \frac12 \int_{\Omega_\eps} |\nabla u|^p\vap \,dx-I_\eps -C\\
 &\ge C_1\int_{\Omega_\eps} |u|^p \vap \,dx-I_\eps-J_\eps-C
  \ge C_1 y_{\eps}^p -I_\eps-J_\eps-C.
 \end{aligned}$$
 Using \eqref{localtheoryBernSZ} and $0\le \varphi\le Cx_n^\alpha$, we deduce that
$|I_\eps| \le C(M_0+1) (\eps^{\alpha-1}+\eps^\alpha t^{-1/p})$ for all $t\in(0,T)$.
This along with the assumption $u\in C(\overline{\R^n_+}\times(0,T))$ and the boundary conditions
guarantees the existence of $\eps_0\in(0,1)$ (depending on $u,t_0,t_1$) such that 
\be{localtheoryboundIeps}
I_\eps(t)+J_\eps(t)\le 1\quad\hbox{ for all $t\in(t_0,t_1) \subset\subset(0,T)$ and $\eps\in(0,\eps_0)$.}
\ee
Consequently,
$$y'_\eps(t)\ge  C_1 y_{\eps}^p -C,\quad t_0<t<t_1,\ \eps\in(0,\eps_0).$$
It follows (cf., e.g.,~\cite[Lemma~9.1(i)]{chso2}) that
 $$y_\eps(t)\le C_2\big[1+(t_1-t)^{-\beta}\big],\quad t_0<t<t_1,\ \eps\in(0,\eps_0).$$
 Letting $\eps\to 0$ and then $t_1\to T$, $t_0\to0$, this implies \eqref{localtheoryupint}.

\end{proof}

\subsection{Pointwise estimates and proof of Proposition~\ref{localtheorythmancient3}}

 \begin{proof}[Proof of Proposition~\ref{localtheorythmancient3}]

 {\bf Step 1.}{\hskip 1mm}{\it Upper estimate for $x_n=1$.} 
 Fix $t\in(0,T)$, $x'\in \R^{n-1}$, $Z=(x',1)$ and set $\hat{B}=B_{1/2}(Z)$.
 Pick any $s\in (t,T)$ and set $\tau=s-t$.
By Theorem~\ref{localtheorythmLY}(ii), we have 
 $$u(Z,t)\le u(y,s)+C\big(\tau^{-\beta}+(1+t^{-1})\tau\big),\quad y\in \hat{B}.$$
Using \eqref{localtheoryupint} in Lemma~\ref{localtheory1L3}, we then obtain
$$\begin{aligned}
  u(Z,t)
  &\le \inf_{y\in \hat{B}} u(y,s)+C\big(\tau^{-\beta}+(1+t^{-1})\tau\big)\\
& \le C\Big(\|u(s)\|_{L^1(\hat{B})}+\tau^{-\beta}+(1+t^{-1})\tau\Big)
  \le C\big(1+(T-s)^{-\beta}+\tau^{-\beta}+(1+t^{-1})\tau\big)
   \end{aligned}$$
  Choose  $s=t+\min\{1,(T-t)/2\}$, and $T-s\ge (T-t)/2$ and $\tau^{-\beta}\le 1+2^\beta(T-t)^{-\beta}$.
We then obtain for $x'\in\R^{n-1},\ 0<t<T$:
 \be{localtheory1l11}
u(x',1,t)\le C(n,p)\big(1+(T-t)^{-\beta}+(1+t^{-1})\min\{1,T-t\}\big).
\ee
 
 \smallskip

 {\bf Step 3.} {\it Proof of \eqref{localtheorycontroludsqa} via scaling argument.}
For given $\lambda>0$, set $u_\lambda(x,t)=\lambda^{\beta-1}u(\lambda x , \lambda^2 t)$.
Since $u_\lambda$ is also a solution of \eqref{localtheoryeqE1} on $(0,\lambda^{-2}T)$ and $C$ in \eqref{localtheory1l11} depends only $n,p$,
we may apply these estimates to $u_\lambda$, which yields
 $$ 
 \begin{aligned}
 \lambda^{\beta-1}u(x',\lambda, t)&=u_\lambda(\lambda^{-1}x',1,\lambda^{-2}t)\le C\Big(1+\lambda^{2\beta} (T-t)^{-\beta}+(1+\lambda^{2} t^{-1})\min\{1,\lambda^{-2}(T-t)\}\Big)\\
 &\le C\Big(1+\lambda^{2\beta} (T-t)^{-\beta}+\min\{\lambda^{2} t^{-1},(T-t)/t\}\Big)
 \end{aligned}
 $$
for all $x'\in\R^{n-1}$, $0< t<T$, $\lambda>0$.
Taking $\lambda=x_n$, we obtain \eqref{localtheorycontroludsqa}.

\smallskip

 {\bf Step 4.} {\it Proof of the gradient estimate \eqref{localtheorycontroludsq1a}.}
Fix $(x,t)\in Q$ and let $R=x_n/2$.
Let $t_0=t/2$ and set $M(t_0)=\underset{B_R(x)\times[t_0,t]} \sup |u|$. 
 Applying the Bernstein type estimate \eqref{localtheorygrabound} to the solution $v(x,s)=u(x,s+t_0)$, defined in $B_R(x)\times(0,t-t_0]$, at $s=t-t_0$,
we have
$$\begin{aligned}
|\nabla u(x,t)|
&\le C\Big(\frac{M(t_0)}{R}+\Big(\frac{M(t_0)}{R^2}\Big)^{1/p}+\Big(\frac{M(t_0)}{t-t_0}\Big)^{1/p}\Big)\\
&\le C\Big(\frac{M(t_0)}{x_n}+x_n^{-\beta}+\Big(\frac{M(t_0)}{x_n}\Big)^{1/p}\Big(\frac{x_n}{t-t_0}\Big)^{1/p}\Big)\\
&\le C\Big(\frac{M(t_0)}{x_n}+x_n^{-\beta}+\Big(\frac{x_n}{t-t_0}\Big)^{\beta}\Big)\le C\Bigg(\frac{M(t_0)}{x_n}+x_n^{-\beta}+(x_nt^{-1})^{\beta}\Bigg)
\end{aligned}$$
where we used Young's inequality in the last estimate.
 By \eqref{localtheorycontroludsqa} we have 
$$
M(t_0)\le C(n,p)
x_n^{1-\beta} \left(1+x_n^{2\beta}(T-t)^{-\beta}+\min\{x_n^2t^{-1},(T-t)/t\}\right),
$$
hence \eqref{localtheorycontroludsq1a}.

 \end{proof}

\noindent{\bf Acknowlegement.} The author thanks Prof.~Philippe Souplet for useful suggestions during the preparation of this work.

\smallskip

\noindent{\bf Statements and Declarations.} The author states that there is no conflict of interest. 
This manuscript has no associated data.

\end{document}